\documentclass[12pt]{article}

\usepackage{amsmath,tabu}
\usepackage{amssymb}
\usepackage{amsthm}

\usepackage{color}         %   produce color text
\usepackage{graphicx}      %   include eps, pdf etc.
\usepackage{changepage}

\usepackage{amssymb,amscd}
\usepackage{amsfonts,amssymb}
\usepackage{amsfonts}      %   AMS math fonts

\usepackage{subfigure}

\usepackage{enumerate}
\usepackage[margin=2.8cm]{geometry}

\newtheorem{lemma}{Lemma}
\newtheorem{proposition}{Proposition}

\newtheorem{theorem}{Theorem}

\theoremstyle{definition}

\newtheorem{remark}{Remark}

\newcommand{\R}{{\mathbb R}}

\newcommand{\pa}{\partial}

\newcommand{\be}{\begin{equation}}
\newcommand{\ee}{\end{equation}}
\newcommand{\beq}{\begin{eqnarray}}
\newcommand{\eeq}{\end{eqnarray}}
\newcommand{\beqq}{\begin{eqnarray*}}
\newcommand{\eeqq}{\end{eqnarray*}}

\begin{document}

\title{\bf Pseudo-peakons and Cauchy analysis for an integrable fifth-order equation of Camassa-Holm type}
\author{
    Enrique G. Reyes$^1$\footnote{email: enrique.reyes@usach.cl}, \quad
    Mingxuan Zhu$^{2}$\footnote{email: mxzhu@qfnu.edu.cn}, \quad and \quad
Zhijun Qiao$^{3}$\footnote{email: zhijun.qiao@utrgv.edu}\\%\quad and\quad
    %Zhaoyang $\mbox{Yin}^{1,3}$ \footnote{email: mcsyzy@mail.sysu.edu.cn}\\
    $^{1}$Departamento de Matem\'atica y Ciencia de la Computaci\'on,\\
     Universidad de Santiago de Chile (USACH)\\ Casilla 307 Correo 2, Santiago, Chile\\
    $^2$School of Mathematical Sciences, Qufu Normal University,\\ Qufu, Shandong 273100, China\\ %\\
   $^{3}$School of Mathematical and Statistical Sciences,\\ The University of Texas Rio Grande Valley, Edinburg,
   TX 78539, USA%\\     $^{4}$Interdisciplinary Research Institute, Faculty of Science\\ Beijing University of Technology,  Beijing 100124, China
}
\date{}
\maketitle
 \baselineskip=15pt
%\begin{document}

%\maketitle

\begin{abstract}
In this paper, we discuss integrable higher order equations {\em of Camassa-Holm (CH) type}.
%The whole CH-type hierarchy 
Our higher order CH-type equations are ``geometrically integrable", that is, they describe one-parametric families 
of pseudo-spherical surfaces, in a sense explained in Section 1, and they are
%is shown 
integrable in the sense of zero curvature formulation ($\simeq$ Lax pair) with infinitely many local conservation 
laws. The major focus of the present paper is on a specific fifth order CH-type equation 
admitting  {\em pseudo-peakons} solutions, that is, weak bounded solutions %which we show bounded %solutions 
with differentiable 
first derivative and continuous and bounded second derivative, but such that any 
higher order derivative blows up. Furthermore, we investigate 
the Cauchy problem of this fifth order CH-type equation on the real line and prove
local well-posedness under the initial conditions $u_0 \in H^s(\R)$, $s > 7/2$.
In addition, we study conditions for global well-posedness in $H^4(\R)$ as well as
conditions causing local solutions to blow up in a finite time. We conclude our paper with some comments on the
geometric content of the high order CH-type equations.
\end{abstract}



\section{Introduction}

In this paper we present a family of fifth-order integrable
equations which generalize ---in a very precise sense discussed below---  the celebrated Camassa-Holm (CH)
equation (see \cite{reyes:CH})
\begin{equation} \label{ch00}
\tilde{m}_{t} = - \tilde{m}_{x}\,u - 2 \,\tilde{m}\,u_{x}\; , \quad \quad \tilde{m} =
u_{xx}-u \; .
\end{equation}
Here and at the beginning of Section 2 we use $\tilde{m}$ instead of $m$ when considering the CH
equation in order to keep the signs used in \cite{reyes:R4}, and not to cause confusion with the conventions
used in the main body of this paper.

Our original motivation for undertaking this project was the observation that the Camassa-Holm
equation can be understood as an Euler equation for the inertia
operator $A_2 =\partial_{xx} - 1$ on the Lie algebra $Vect(S^1)$
of the (Fr\'echet) Lie group $Diff(S^1)$, and that, in turn, the latter equation can
be interpreted geometrically as determining geodesics for a
Riemannian metric on $Diff(S^1)$ defined via $A_2\,$, see for instance \cite{A,KW,M}, and
also \cite{GPoR}\footnote{A similar observation is valid for a group of diffeomorphisms of the line that
is slightly more difficult to describe, see \cite{C_fourier}.}.  Now, the operator
$A_2$ appears naturally in the zero curvature representation of CH, see Equations
(\ref{reyes:l40}) and (\ref{reyes:l4}) below. A rather obvious question that arises
is whether considering different inertia operators $A_n$ in (\ref{reyes:l40}) and
(\ref{reyes:l4}) would yield equations that may be of interest for the theory
of fluids and may determine interesting geometry on diffeomorphism groups.

We can think of two immediate objections to considering this question. First:  
Wouldn't changing just the inertia operator used in (\ref{reyes:l40}) and
(\ref{reyes:l4}) yield equations ``equivalent" to CH? Second: Wouldn't we obtain only
uninteresting equations, since in the paper \cite{CP} Constantin and Kolev proved  that
in the case of $Vect(S^1)$ and the (Fr\'echet) Lie group $Diff(S^1)$ the only
{\em bi-Hamiltonian} equations (for some natural Poisson structures) arising from inertia operators
$A_n = \sum (-1)^k \partial_x^{2k}$ are the inviscid Burgers equation ($k=0$) and the Camassa-Holm
equation ($k=1$)?

{
The way we have found of replying to these objections is through geometric considerations.
It was proven in \cite{reyes:R4} that the Camassa-Holm (\ref{ch00}) describes pseudo-spherical
surfaces, in the following sense:

\noindent Define three one-forms $\omega^i$, $i=1,2,3$, as follows:
\begin{eqnarray}
\omega^1 & = & \left(\frac{1}{2}\lambda + \frac{1}{2\lambda} + \tilde{m}  \right) dx +
  \left(- \frac{1}{2\lambda} u+ \frac{u}{2}\lambda - u \tilde{m} - \frac{1}{2} - \frac{\lambda^2}{2} \right) dt
  \label{o1}\\
\omega^2 & = & - u_x dt \label{o2}\\
\omega^3 & = & \left(- \frac{1}{2}\lambda + \frac{1}{2\lambda} - \tilde{m} \right) dx +
 \left(- \frac{1}{2\lambda} u - \frac{u}{2}\lambda + u \tilde{m} - \frac{1}{2} + \frac{\lambda^2}{2} \right) dt\; .
 \label{o3}
\end{eqnarray}
Then, the structure equations of a pseudo-spherical surface,
\begin{equation}
d \omega^1 = \omega^3 \wedge \omega^2\; , \quad \quad d \omega^2 = \omega^1 \wedge \omega^3\; , \quad \quad
d \omega^3 = \omega^1 \wedge \omega^2\; , \label{o4}
\end{equation}
are satisfied if and only if $u(x,t)$ and $\tilde{m}(x,t)$ satisfy the Camassa-Holm equation\footnote{Generally speaking, we say that a scalar differential equation $\Delta=0$ for a function $u(x,t)$ describes pseudo-spherical surfaces if there exist one-forms $\omega^i$, $i=1,2,3$, depending on
finite jets of $u(x,t)$, such that Equations (\ref{o4}) hold on solutions to $\Delta=0$, see
\cite{reyes:ChT}. If the one-forms $\omega^i$ depend non-trivially on a parameter, we say that the equation
$\Delta=0$ is {\em geometrically integrable}, see \cite{reyes:ChT,reyes:R4}.}. The existence
of this geometric structure is quite useful for applications: it allows us to obtain quadratic pseudo-potentials, a
zero curvature formulation for CH, a sequence of non-trivial local conservation laws, non-local symmetries, and a
geometrically motivated modified Camassa-Holm equation, see \cite{reyes:R4,HHR,G-R}; see also \cite{s-dual,GH,
Qiao-CMP} for independent articles showing zero curvature representations of CH as well. 

Now, we can modify the one-forms (\ref{o1})-(\ref{o3}) in such a way that the
structure equations (\ref{o4}) are still equivalent to a scalar equation, by keeping the
$\lambda$ dependence unchanged and modifying the coefficients of the Laurent polynomials in
$\lambda$ appearing in (\ref{o1})-(\ref{o3}). We would then say that {\em the new scalar
equation generalizes CH}. Certainly there are restrictions in order for this plan to make sense:
there is no point in changing $u$ for a differentiable function of $u$, for example, since in this case we would 
get an equation obviously equivalent to CH. {\em But, if we consider modifications of $(\ref{o1})$-$(\ref{o3})$ 
depending on
higher order jets of $u$, we would obtain ---in principle--- new equations, not related to CH by
invertible local transformations.}

But perhaps we can change $u$ in completely arbitrary
ways, and therefore our idea is just too general and naive to be of interest? No. Hereafter, we restrict 
ourselves to fifth order equations; the following geometric theorem, to be proven in a separate paper, states 
precisely how we can modify the coefficients of the Laurent polynomials in $\lambda$  appearing in
(\ref{o1})-(\ref{o3}) in order for the structure equations (\ref{o4}) to be equivalent to a
scalar differential equation.

\begin{theorem}
The most general fifth order equation of pseudo-spherical type that is equivalent to the
structure equations $(\ref{o4})$ with associated one-forms
\begin{eqnarray*}
\omega^1 & = & \left(a \lambda + \frac{d}{\lambda} + F_{11} \right)dx +
\left(\frac{1}{\lambda} G_{11} + G_{12} +
\lambda G_{13} + \lambda^2 G_{14} \right) dt \\
\omega^2 & = & H(u,u_x,u_{xx},u_{xxx}) dt \\
\omega^3 & = & \left(-a \lambda + \frac{d}{\lambda} + F_{31} \right)dx +
\left(\frac{1}{\lambda} G_{31} + G_{32} +
\lambda G_{33} + \lambda^2 G_{34} \right) dt \; ,
\end{eqnarray*}
in which the functions $F_{ij}$ and $G_{ij}$ depend on $u,u_x,u_{xx},u_{xxx},u_{xxxx}$, is
\begin{eqnarray}
4 a d D_t G_{13} - D_t D_x^2 G_{13} &= &
 \frac{2}{a} (D_xG_{13})(D_x^2 G_{13}) + \frac{1}{a} (G_{13} + G_{131}) D_x^3 G_{13} \label{se4} \\
&  &  - 12 d G_{13} D_x G_{13} - 8 d G_{131} D_x G_{13}\; , \nonumber
\end{eqnarray}
%in which 
where $G_{13}$ is an arbitrary function of $u, u_x, u_{xx}$ satisfying $G_{13,u_{xx}} \neq 0$, $D_x$ indicates total derivative with respect to $x$, $G_{131}$ is a constant number, and $a\not=0,d$ are two  arbitrary real constants.

\noindent The one-forms $\omega^i = f_{i1} dx + f_{i2} dt$ associated to Equation $(\ref{se4})$ are
determined by
\begin{eqnarray*}
f_{11} & = & {a}{\lambda} + \frac{d}{\lambda} + \frac{1}{2 d G_{34}}\left[ D_x^2 G_{13}
- 4 a d G_{13} - 2 a d G_{131} \right] \\
f_{12} & = & -\frac{d}{a \lambda}(G_{13}+G_{131}) - \frac{1}{2 a d G_{34}}(G_{ 13} + G_{131})D_x^2 G_{13} + \frac{1}{G_{34}}(G_{13}+G_{131})(2G_{13}+G_{131})  \nonumber \\
   &   &  ~  - \frac{d G_{34}}{a} + G_{13}\lambda - G_{34}\lambda^2  \\
f_{21} & = & 0 \\
f_{22} & = & - \frac{1}{a} D_x G_{13} \\
f_{31} & = & -{a}{\lambda} + \frac{d}{\lambda} - \frac{1}{2 d G_{34}}\left[ D_x^2 G_{13}
- 4 a d G_{13} - 2 a d G_{131} \right] \\
f_{32} & = & -\frac{d}{a \lambda}(G_{13}+G_{131}) + \frac{1}{2 a d G_{34}}(G_{ 13} + G_{131})D_x^2 G_{13} -\frac{1}{G_{34}}(G_{13}+G_{131})(2G_{13}+G_{131})  \nonumber \\
   &   &  ~  - \frac{d G_{34}}{a} - G_{13}\lambda + G_{34}\lambda^2
\end{eqnarray*}
where $G_{34}$ is, if fact, a non-zero constant number.
\end{theorem}

We say that the large non-trivial family of fifth order equations of geometric origin (\ref{se4}) are 
``higher order CH-type equations", or, that they are ``fifth order manifestations" of CH. 

Let us now answer the objections raised
at the beginning of this section. Equations (\ref{se4}) {\em are not} equivalent to CH, as they
are only related to CH by non-invertible differential substitutions:
if we set $v = G_{13}(u, u_x, u_{xx})$ in (\ref{se4}), we find that $v$ satisfies (up to some parameters) the
CH equation, { but there is no reason to expect that the ``internal" properties of (\ref{se4}) and 
CH are the same}. Also, the existence of the one-forms $\omega^i$ satisfying (\ref{o4}) allow us
to check using \cite{reyes:ChT} that Equations (\ref{se4}) admit zero curvature representations.
}
The existence of these zero
curvature representations imply the integrability of our
corresponding ``higher order CH-type equations", not in the sense of
being bi-Hamiltonian, but in the sense of possessing an infinite
number of non-trivial local conservation laws. We do lose the
interpretation of our higher order CH-type equations as standard Euler equations but, as we show
in the last section of this work, the equations considered here can also be connected with the geometry of $Diff(S^1)$.

\smallskip

{\color{black} In this paper we consider the following fifth
order equation of CH-type arising from $(\ref{se4})\,$:}
\begin{eqnarray}
\left\{\begin{array}{l}
m_{t}=-m_xv-2mv_x, \\ %=k_1\left[m\partial^{-1}_{x}mb_{-2,x}\right]_x+\frac{1}{2}k_2(2m
v=u-u_{xx}=(1-\partial^2)u=-A_2(u),\\
m=v-v_{xx}=(1-\partial^2)^2u=A_4(u)\; . %b_{-2,x}+m_xb_{-2}),\\.%m_y=b_{-2,x}-b_{-2,xxx}.
\end{array}\right.
\label{FOCHInt1}
\end{eqnarray}
The expansion of (\ref{FOCHInt1}) reads
\begin{eqnarray}
u_{t}-2u_{{xxt}}+u_{xxxxt} &=& -3uu_{x}+4uu_{xxx}-uu_{xxxxx}+5u_{x}u_{xx}-2u_{x}u_{xxxx} \nonumber\\
& & -6u_{xx}u_{{xxx}}
+2u_{xxx}u_{xxxx}
+u_{xx}u_{xxxxx}\; .
\label{ch410}
\end{eqnarray}

This equation is obtained from (\ref{se4}) if we choose $G_{13} = u - u_{xx}$, $a=1$, $d=1/4$ and
$G_{131} =0$. { We decided to make this choice partially motivated by \cite{LQ,ZCJQ} where the
authors study the model }
\begin{equation} \label{foch0}
m_{t} + u m_x + b u_x m = 0\; , \quad \quad m = (1 - \alpha^2 \partial_x^2)(1-\beta^2 \partial_x^2) u\; 
.
\end{equation}
We consider Equation (\ref{FOCHInt1}) in Section 2, where we show that it admits a Lax pair, as anticipated in 
this section, and that it possesses an infinite number of non-trivial local conservation laws; we also present
its Lie algebra of classical symmetries and comment on (the no existence of) nonlocal
symmetries of a specific type\footnote{We insist: Equation
(\ref{FOCHInt1}) (or (\ref{ch410})) is related to the standard CH equation
\begin{equation}
\overline{m}_{t} = - \overline{m}_{x}\,v - 2 \,\overline{m}\,v_{x}\; , \quad \quad \overline{m} = v-v_{xx} \; . \label{ch-int}
\end{equation}
by a differential substitution (the transformation $v = (1-\partial^2)u$, $\overline{m} = (1-\partial^2)m$), but 
this transformation is not invertible as a mapping between (subsets of) finite order jet bundles, and so Equations
(\ref{ch-int}) and (\ref{FOCHInt1}) are not equivalent. We can study some aspects of the former
equation using results on the latter one (see for instance our remarks at the end of Section 3),
but these two equations are different in essential ways, as our computations of symmetries and conservation laws
show. We will highlight differences and similarities between (\ref{ch-int}) and (\ref{FOCHInt1}) throughout the 
paper.}. Then, {\em in Section $3$  we show that Equation $(\ref{FOCHInt1})$
possesses pseudo-peakons, this is, bounded weak travelling wave solutions with continuous first derivative and 
continuous second derivative, but whose higher order derivatives blow up}. We claim that the existence of this
kind of weak solutions, even more than the integrability characteristics of (\ref{FOCHInt1}),
is what makes our equations (\ref{se4}) ---and the particular case (\ref{ch410})--- of interest. {\color{black} 
Pseudo-peakons seem to be fairly new objects in the nonlinear landscape; they have been observed before only in the
``fifth order Camassa-Holm equation" (\ref{foch0}) studied in \cite{LQ,ZCJQ}. It is quite intriguing to find
them at the level of Equation (\ref{se4})}\footnote{ We remark that the Camassa-Holm equation does not have pseudo-peakon weak solutions, as it follows from Lenells' study of travelling wave solutions to CH, see \cite[Theorem 1]{L}. Thus, their existence implies, once again, that CH and (\ref{FOCHInt1}) are different equations.}. In Section 4 we discuss local well-posedness of (\ref{FOCHInt1}) for
initial conditions $u_0 \in H^s(\R)$, $s > 7/2$, which is to be contrasted with the corresponding result for
standard CH equation: for instance, in \cite[Theorem 3.1]{Guill} local well-posedness of the Camassa-Holm equation
is proven for initial data in $H^s(\R)$, $s > 3/2$. Thus, the ``good" spaces of initial data for
CH and (\ref{FOCHInt1}) are different. In this section we also prove a theorem on global
well-posedness of (\ref{FOCHInt1}) in $H^4(\R)$, and we present conditions causing local solutions to blow up in
a finite time\footnote{We note, see Subsection 4.2, that the blow-up mechanism of (\ref{FOCHInt1}) is not the same as in the CH case: for the CH equation\cite{CE,CE1}, blow-up phenomenon means that the first-order derivative tends to negative infinity.
However, for our equation blow-up means that the third-order derivative tends to infinity.}.
Finally, in Section 5 we collect several general remarks. We introduce a hierarchy of higher order Camassa-Holm type equations to which (\ref{FOCHInt1}) belongs, we observe that (in the periodic case, $x \in S^1$) these equations are indeed constructed via inertia operators, that our equations can be written in terms of the geometry of the loop group $Diff(S^1)$, and finally we go back to their relation with classical theory of surfaces.

\section{Equations of Camassa--Holm type}

Let us begin by recalling the following observation about the
important Camassa--Holm equation introduced in \cite{reyes:CH}.

\begin{theorem}
The compatibility condition of the linear problem
$$
d \psi = (X dx + T dt)\psi \; ,
$$
where %in which
$\psi = (\psi_{1}, \psi_{2})^{t}$,
\begin{equation}  \label{reyes:l40}
X = \frac{1}{2} \left [
\begin {array}{cc}
  0 &  \lambda + 2\,\tilde{m}  \\
  \lambda^{-1} & 0
\end {array}
\right ] \; ,
\end{equation}
and
\begin{equation}
 T = \frac{1}{2} \left [
\begin {array}{cc}
\displaystyle
- {u_{x}} & - 2 u  \, \tilde{m} + \lambda u  - \lambda^{2} \\
- 1 - u \lambda^{-1} & u_{x}
\end {array}
\right ]  ,                       \label{reyes:l4}
\end{equation}
is the Camassa-Holm $(CH)$ equation
\begin{equation}
\tilde{m}_{t} = - \tilde{m}_{x}\,u - 2 \,\tilde{m}\,u_{x}\; , \quad \quad \tilde{m} = u_{xx}-u \; .\label{reyes:ch}
\end{equation}
\end{theorem}

This theorem appears in \cite{s-dual} and \cite{reyes:R4}. As indicated in Section 1, we are using $\tilde{m}$
instead of $m$ here in order to keep the signs used in \cite{reyes:R4} when discussing CH.

It is well-known how to obtain quadratic
pseudo-potentials and conservation laws from associated $sl(2,\mathbb{R})$-valued linear
problems, see \cite{reyes:ChT} and Refs. \cite{er_cl,reyes:R4} for full details.
We obtain:

\begin{theorem} \label{reyes:quadratic}
The CH equation $(\ref{reyes:ch})$ admits a quadratic
pseudo-poten\-tial $\gamma$ determined by the compatible equations
\begin{equation}
\tilde{m} = \gamma_{x} + \frac{1}{2\lambda} \, \gamma^{2} - \frac{1}{2}\lambda\; , \qquad
\gamma_{t} = \frac{\gamma^{2}}{2}\left[1+\frac{1}{\lambda}u\right] -
u_{x}\gamma -u\,\tilde{m} + \left[\frac{1}{2}u\lambda  -
\frac{1}{2}\lambda^{2}\right]\, , \label{reyes:tpart3}
\end{equation}
where %in which
 $\lambda \neq 0$ is a parameter.  Moreover, Equation
$(\ref{reyes:ch})$ possesses the parameter-dependent conservation law
\begin{equation}
\gamma_{t}  =  \lambda \left(u_{x} - \gamma - \frac{1}{\lambda}
u\gamma\right)_{x} . \label{reyes:cl3}
\end{equation}
\end{theorem}

We can interpret Equations (\ref{reyes:tpart3}) and (\ref{reyes:cl3}) as determining a
``Miura-like" transformation and a ``modified Camassa-Holm" (MOCH) model. These observations are developed
in \cite{G-R}.

We use (\ref{reyes:tpart3}) and (\ref{reyes:cl3}) to construct
conservation laws for the CH equation. Setting $\gamma =
\sum\limits_{n=1}^{\infty} \gamma_{n} \lambda^{n/2}$ and substituting
this expansion into (\ref{reyes:tpart3}), we find the conserved densities
\begin{gather}
\gamma_{1} =  \sqrt{2}\,\sqrt{\tilde{m}}\; , \qquad \gamma_{2}  =   -
\frac{1}{2}\,\ln (\tilde{m})_{x}\; , \qquad \gamma_{3}  =
\frac{1}{2\sqrt{2}\,\sqrt{\tilde{m}}}\left[ 1 - \frac{\tilde{m}_{x}^{2}}{4\,\tilde{m}^{2}}
+
\ln (\tilde{m})_{xx}\right] , \label{reyes:cd-1} \\
\gamma_{n+1}  =  - \frac{1}{\gamma_{1}} \, \gamma_{n,x} -
\frac{1}{2\gamma_{1}}  \sum_{j=2}^{n} \gamma_{j} \, \gamma_{n+2-j} \; , \qquad n \geq 3\; , \label{reyes:cd-4}
\end{gather}
while the expansion $\gamma = \lambda + \sum\limits_{n=0}^{\infty}
\gamma_{n} \lambda^{-n}$ implies
\begin{equation}
\gamma_{0,x} + \gamma_{0} = \tilde{m}\; , \qquad \gamma_{n,x} + \gamma_{n}
= - (1/2) \sum_{j=0}^{n-1} \gamma_{j} \, \gamma_{n-1-j}\; , \quad
                                n \geq 1\; .  \label{reyes:cd2}
\end{equation}
It is shown in \cite{reyes:R4} that the local conserved densities
$\gamma_{n}$ determined by (\ref{reyes:cd-1}) and
(\ref{reyes:cd-4}) correspond to the ones found by Fisher and
Schiff in \cite{reyes:s-asso} by using an ``associated
Camassa--Holm equation'', while (\ref{reyes:cd2}) generates %yields
the local conserved densities $u$, $u_{x}^{2} + u ^{2}$ and $u u_{x}^{2} +
u ^{3}$, and a sequence of nonlocal conservation laws. {\color{black} We note that the densities
$u_{x}^{2} + u ^{2}$ and $u u_{x}^{2} + u ^{3}$ determine a pair of compatible Hamiltonian operators
associated to CH, see \cite{reyes:CH} and also \cite{reyes:s-asso}. }

\medskip

Let us now begin the study of our higher order equation of Camassa-Holm type (\ref{ch410}). First of all we
observe that in terms of associated linear problems
(instead of one-forms as in Section 1), if we consider the matrices
%our basic idea is to use
%$m = A_n (u)$ for higher order operators $A_n$ in the matrices (\ref{reyes:l4}), instead of using simply
%$m = (1-\partial_{xx})u\,$: we keep the $\lambda$-pole structure of the matrices $X$ and $T$
%appearing in ({reyes:l40}) and (\ref{reyes:l4}), but we write $m = A_n (u)$ in
%$X$ and we modify $T$ so that the zero curvature equation $X_t -
%T_x + [X,T] = 0$ is equivalent to a scalar partial differential
%equation. Here we work out the case in which $A_n$ is a fourth
%order operator, and we present a general construction in
%the final section of this paper.
%
%\smallskip
%
%We choose $A_4 (u) = -u_{xxxx}+2\,u_{xx}-u$ and we select
\begin{equation} \label{xch40}
X_4 = \left[ \begin {array}{cc} 0&
\frac{1}{2}\,\lambda-u_{xxxx}+2\,u_{xx}-u\\\noalign{\medskip}\frac{1}{2}\,{\lambda}^{-1}&0\end
{array} \right]
\end{equation}
and
\begin{equation} \label{tch40}
T_4 = \frac{1}{2}\,\left[ \begin {array}{cc} -u_{x}+ u_{xxx}&
 \left( -2\,u+2\,u_{xx} \right)  \left( -u_{xxxx}+2\,
 u_{xx}-u \right) + \lambda \left( u-u_{xx}
 \right) - {\lambda}^{2}\\\noalign{\medskip}\displaystyle - 1-
{\frac {u}{\lambda}}+ {\frac {u_{xx}}{\lambda}}& u_{x}-
u_{xxx}\end {array} \right] \; .
\end{equation}
A straightforward computation shows that the equation
$$
X_{4,t} - T_{4,x} + [X_4 , T_4] = 0
$$
is equivalent to Equation (\ref{ch410}).   % or (\ref{FOCHInt1})
%\begin{gather}
%-3\,u_{x}u+5\,u_{x}u_{xx}-u_{t}-u_{xxxxt}-6\,
%u_{{xxx}}u_{xx}+2\,u_{{xxt}}\nonumber \\
%+\,4\,uu_{xxx} -2\,u_{x}
%u_{xxxx}-u_{xxxxx}u+u_{xxxxx}u_{xx}+2\,u_{xxx}u_{xxxx} = \, 0 \; ,
%\label{ch41}
%\end{gather}
%which can be rewritten in a concise form as follows:
%\begin{eqnarray}
%\left\{\begin{array}{l}
%m_{t}+m_xv+2mv_x=0, \\ %=k_1\left[m\partial^{-1}_{x}mb_{-2,x}\right]_x+\frac{1}{2}k_2(2m
%m=v-v_{xx},\\%b_{-2,x}+m_xb_{-2}),\\
%v=u-u_{xx}.%m_y=b_{-2,x}-b_{-2,xxx}.
%\end{array}\right.
%\label{HOCHInt1}
%\end{eqnarray}
Thus, the geometrically integrable Equation (\ref{ch410}) ---we recall that this notion was introduced in footnote 
2--- is integrable in the sense of admitting the parameter-depending Lax pair
$\psi_x= X_4 \psi, \psi_t= T_4 \psi$ and (as we will see momentarily) of possessing an infinite number of
non-trivial local conservation laws; henceforth we call either (\ref{ch410}) or (\ref{FOCHInt1}) the integrable 
fifth-order CH-type equation.

\begin{remark}
We are aware that the foregoing construction and the comments made in footnote 3 imply that
% Equation (\ref{ch410}) 
%is gauge-equivalent to the standard CH equation, 
%via $\tilde{m}\to -(1-D_x^2)m,\  u\to (1-D_x^2)u$, $m=(1-\partial^2)u$, 
%and that 
the only difference between the matrices (\ref{xch40}) and  (\ref{reyes:l40}) is the potential function. 
Thus, {\em from the point of  view of scattering/inverse scattering}, our equation is a ``fifth-order 
manifestation" of CH, as we said in Section 1. This means that 
 the properties of our equation that depend only on the pole structure of matrices (\ref{xch40}) and
 (\ref{reyes:l40}) can be trivially found from the corresponding properties of the CH equation. However, the inverse
scattering transform (IST) method is not the only way to study nonlinear equations.  As anticipated in footnotes
3, 4, and 5, the structures of (local/nonlocal) symmetries and conservation laws
are different for CH and (\ref{ch410}), their {\em weak} solutions are also different (in Section 3 we obtain {\em pseudo-peakons} instead of peakons), and their analytic properties are
different (in Section 4 we show, for instance, that their blow-up mechanisms differ).
\end{remark}

\smallskip

We now present conservation laws and symmetries of Equation (\ref{ch410}) in an explicit form.

\bigskip

\noindent {\large {\bf Conservation laws} }

\noindent After the classical work \cite{WE} (and the geometric reinterpretation of \cite{WE}
appearing in \cite{reyes:ChT,er_cl}), we compute conservation laws
using quadratic pseudo-potentials. {\color{black} We use Theorem \ref{reyes:quadratic} and Equations 
(\ref{reyes:cd-1})-(\ref{reyes:cd2}).
%
%\begin{proposition}
%Let us assume that a given equation $\Xi(x,t,u,\dots) = 0$ is the
%integrability condition of an $sl(2,\mathbb{R})$-valued linear
%problem $\Psi_x = X \Psi$ and $\Psi_t = T \Psi$, in which the
%matrices $X=(X_{ij})$ and $T=(T_{ij})$ depend on $x,t,u$, finite
%numbers of derivatives of $u$, and (possibly) a parameter
%$\lambda$. Then, the following pair of Riccati equations determines
%a quadratic pseudo-potential for $\Xi = 0:$
%\begin{eqnarray*}
%- 2\, \Gamma_x & = & (-X_{12}+X_{21}+2\, X_{11}) - 2\, \Gamma
%(X_{12} +
%X_{21}) + \Gamma^2 (-X_{12}+X_{21}- 2\, X_{11})  \\
%- 2\, \Gamma_t & = & (-T_{12}+T_{21}+2\, T_{11}) - 2\, \Gamma
%(T_{12} + T_{21}) + \Gamma^2 (-T_{12}+T_{21}- 2\, T_{11})  \; .
%\end{eqnarray*}
%Moreover, the equation $\Xi = 0$ admits a conservation law with
%conserved density
%\begin{equation*}
%X_{12} + X_{21} - \Gamma (-X_{12} + X_{21} - 2\, X_{11})
%\end{equation*}
%and flux
%\begin{equation*}
%T_{12} + T_{21} - \Gamma (-T_{12} + T_{21} - 2\, T_{11}) \; .
%\end{equation*}
%\end{proposition}
%
%\smallskip
%
%\smallskip
%
%In the case of the linear problem determined by (\ref{xch40}) and
%(\ref{tch40}), it is convenient to apply a gauge transformation to
%the connection $X_4 dx + T_4 dt$ with gauge matrix
%\begin{equation} \label{gauge}
%R = \left[ \begin{array}{cr} 1 & -1 \\ 1 & 1 \end{array} \right]
%\; .
%\end{equation}
%We perform this transformation and then we use Proposition 1.
Making the substitution  $\tilde{m}\to -(1-D_x^2)m,\  u\to (1-D_x^2)u$, $m=(1-\partial^2)u$, we obtain that }
Equation (\ref{ch410}) admits the quadratic pseudo-potential
\begin{eqnarray}
{\frac {\partial \Gamma}{\partial x}}
 & = & \frac{1}{2}\,\lambda-u_{xxxx}+2\,u_{xx}-u- {\frac {
\Gamma^2}{2\,\lambda}} \label{g1} \\
 {\frac {\partial \Gamma}{\partial t}} & = & -\, \left( \,
u-\,u_{xx} \right)  \left( -u_{xxxx}+2\,u_{xx}-u \right)
+\frac{1}{2}\,\lambda\, \left( u-u_{xx} \right) -\frac{1}{2}\,{
\lambda}^{2} \nonumber \\
 &  & -\, \Gamma \left( u_{x}-u_{xxx}
 \right) +\frac{1}{2}\, \Gamma^{2} \left( 1+{
\frac {u}{\lambda}}-{\frac {u_{xx}}{\lambda}} \right)
\end{eqnarray}
and the parameter dependent conservation law
\begin{equation} \label{g2}
\left( \frac{\Gamma}{\lambda} \right)_t = \left( u_{x}-u_{xxx}-
\Gamma \left( 1+{\frac {u}{\lambda}}-{\frac {u_{xx}}{\lambda}}\right) \right)_x \; .
\end{equation}
Expansion in powers of $\lambda$ as in
(\ref{reyes:cd-1})--(\ref{reyes:cd2}) yields a sequence of non-trivial local conservation laws.
{We write them down in detail using the concise form (\ref{FOCHInt1}) of the CH-type equation (\ref{ch410})
that we introduced in Section 1, this is,
$$
m_t+2 v_x m + v m_x =0 \; ,
$$
where
$v = u - u_{xx}$ and $m = v - v_{xx} = u - 2 u_{xx} + u_{xxxx}$.
Setting $\Gamma = \sum_{n=1}^\infty \gamma_n \lambda^{n/2}$ and replacing into (\ref{g1}) we obtain
\begin{eqnarray}
\gamma_1 & = & \sqrt{2} \sqrt{-m} \label{cd1} \\
\gamma_2 & = & \frac{-\gamma_{1,x}}{\gamma_1} = \frac{1}{2}\ln(|-m|)_x \label{cd2} \\
\gamma_3 & = & \frac{1}{2\sqrt{2}\sqrt{-m}} \left( 1 - \frac{1}{4}\frac{m_x^2}{m^2} - \ln(|-m|)_{xx} \right)
                                                                                                  \label{cd3} \\
\gamma_n & = & \frac{-1}{\gamma_1} \gamma_{n,x} - \frac{1}{\gamma_1} \sum \gamma_k \gamma_{n+2-k}\; , \quad \quad
n \geq 3\; .\label{cd4}
\end{eqnarray}

\noindent Now we set $\Gamma = \lambda + \sum_{n=0}^\infty \tilde{\gamma}_n \lambda^{-n}$ and we
substitute into (\ref{g1}). We obtain
\begin{eqnarray}
\tilde{\gamma}_{0,x} + \tilde{\gamma}_0 & = & - m \label{cd5} \\
\tilde{\gamma}_{n,x}+\tilde{\gamma}_n & = & -\frac{1}{2}\sum_{k=0}^{n-1}\tilde{\gamma}_k\tilde{\gamma}_{n-1-k}\; .
\label{cd6}
\end{eqnarray}
Equation (\ref{cd5}) yields the conserved density $\tilde{\gamma}_0 = - u_{xxx}+u_{xx} + u_x - u$. In order to
find further densities we note that Equation (\ref{g2}) means not only that the functions $\tilde{\gamma}_n$ are
conserved densities, but also that so are the functions $\tilde{\gamma}_n + \tilde{\gamma}_{n,x}$. From
(\ref{cd6}) we obtain
\begin{equation} \label{cd7}
\int ( \tilde{\gamma}_1 + \tilde{\gamma}_{1,x})dx = - \frac{1}{2} \int \tilde{\gamma}_0^2 dx = - \frac{1}{2}
\int (u_{xxx}^2 + 3 u_{xx}^2 + 3 u_x^2 + u^2 ) dx \; ,
\end{equation}
in which we have eliminated total derivatives and we have also eliminated all boundary terms that appear after
using integration by parts. Thus,
\begin{equation} \label{cd8}
H_1 = u_{xxx}^2 + 3 u_{xx}^2 + 3 u_x^2 + u^2
\end{equation}
is a local conserved density for (\ref{ch410}). This density will be important for our analysis of the global
well-posedness of (\ref{ch410}), see Theorems 6, 7, and 8 below. Further conserved densities arising from
(\ref{cd6}) are non-local expressions. For example, taking $n=2$ in (\ref{cd6}) we find
$$
\int ( \tilde{\gamma}_2 + \tilde{\gamma}_{2,x})dx = - \int \tilde{\gamma}_0 \tilde{\gamma}_1 dx\; ,
$$
this is, after integration by parts,
\begin{eqnarray}
\int ( \tilde{\gamma}_2 + \tilde{\gamma}_{2,x})dx & = & - \int (u_{xx} - u_x)\tilde{\gamma}_{1,x}dx +
                                          \int u (\tilde{\gamma}_{1,x} + \tilde{\gamma}_1) dx \nonumber \\
 & = &  \int(u_{xxx}-u_{xx})\tilde{\gamma}_{1}dx - \frac{1}{2}\int u (- u_{xxx}+u_{xx} + u_x - u)^2 dx\;
 \label{j}.
\end{eqnarray}

\medskip

It is immediate that all conserved densities $\gamma_k$, $k$ odd, which are included in (\ref{cd1})--
(\ref{cd4}) are non-trivial. Indeed, it is known that all conserved densities $\gamma_k$, $k$ odd, which appear in
(\ref{reyes:cd-1})--(\ref{reyes:cd-4}) are non-trivial: they have a term that depends only on (in our notation) $\tilde{m}$. Thus, the odd-label densities appearing in (\ref{cd1})--(\ref{cd4}) have a term that depends only
on $-(1-\partial^2)m$, hence they are non-trivial as well.

\begin{remark}
The foregoing computations tell us once again that Equation (\ref{ch410}) has different properties than the
standard Camassa-Holm equation.
In fact, while expansion in powers of $\lambda^{n/2}$ of the density $\Gamma$
of (\ref{g2}) yields conservation laws that correspond to the ones appearing in
(\ref{reyes:cd-1})--(\ref{reyes:cd-4}), we lose one local density if we expand in powers of $\lambda^{-n}\,$.
In the case of the fifth order CH-type equation (\ref{ch410}), we obtain (\ref{j}) instead of a density similar to
the conserved density $H_2 =u (u_{x}^{2} + u ^{2})$ arising in the Camassa-Holm case, see (\ref{reyes:cd2}) and
\cite{reyes:R4}. {\color{black} This is important, because, as discussed after Equation (\ref{reyes:cd2}), $H_2$ 
is one of the CH Hamiltonian densities, see for
example \cite[Equation (2)]{reyes:s-asso}. In other words, we cannot translate the CH bi-hamiltonian
structure to our Equation (\ref{ch410}). Of course, this is in agreement with the Constantin-Kolev
result in \cite{CP} on the classification of bi-hamiltonian equations that we mentioned in Section 1.}
\end{remark}

\begin{remark}
The CH-type equation (\ref{ch410}) admits a Hamiltonian formulation. Indeed, we consider the density $H_1$
given in (\ref{cd8}) and define
$$
\mathcal{H} = \int H_1 dx = \frac{1}{2} \int \left( u_{xxx}^2 + 3 u_{xx}^2 + 3 u_x^2 + u^2 \right) dx \; .
$$
Then, it is straightforward to check that Equation (\ref{ch410}) is equivalent to the Hamilton equation
$$
m_t = -(\partial_x m + m \partial_x) \frac{\delta \mathcal{H}}{\delta m} =
     -(\partial_x\, m + m\, \partial_x)\,(1-\partial_x^2)^{-2} \frac{\delta \mathcal{H}}{\delta u} \; ,
$$
in which, as usual in this paper, $m = (1-\partial_x^2)^2 u = u - 2 u_{xx} + u_{xxxx}$ and $\delta/\delta u$ is the
standard Euler operator.
\end{remark}

\bigskip

\noindent {\large {\bf Symmetries} }

\noindent Now we compute symmetries for (\ref{ch410}). The Lie algebra of point symmetries of
this equation is much richer than the Lie algebra of point symmetries of the CH equation: we obtain the following
result with the help of GeM, see \cite{Ch}, and the MAPLE built-in package PDEtools:

\begin{proposition}
 ~

\begin{itemize}
\item The Lie algebra of point symmetries of Equation $(\ref{ch410})$ is generated by the vector fields
$$
V_1=\frac{\partial}{\partial x}\; , \quad \quad \quad V_2= \frac{\partial}{\partial t}\; , \quad \quad \quad
V_{3F}= F(t) e^x \frac{\partial}{\partial u}\;
$$
and
$$
V_{4G}=G(t) e^{-t} \frac{\partial}{\partial u}\; , \quad \quad \quad
V_5 = t \frac{\partial}{\partial t} - u \frac{\partial}{\partial u}\; ,
$$
in which $F(t)$ and $G(t)$ are arbitrary smooth functions.
\item The Lie algebra structure of point symmetries of Equation $(\ref{ch410})$ is determined by the
commutator table

$$
%\begin{center}
\begin{tabu}{ c||c|c|c|c|c|}
% \hline
    & V_1 & V_2 & V_{3F} & V_{4G} & V 5 \\
    \hline \hline
 V1 & 0 & 0 & V_{3F} & - V_{4G} & 0 \\
 \hline
V2  &  & 0 & V_{3F_t} & V_{4G_t} & V_{2}  \\
\hline
V3  &  &  & 0 & 0 & -V_{3\,(-t F_t+F)}  \\
\hline
V4  &  &  &  & 0 & -V_{4(tG_t +G)}  \\
\hline
V5  &  &  &  &  & 0  \\
 \hline
\end{tabu}
%\end{center}
$$
\end{itemize}
\end{proposition}

\bigskip

The existence of symmetries $V_1$ and $V_2$  prompts us to look for solutions of the form $u(x,t) = f(x + ct)$.
We easily find $u(x,t) = A\, {\mathrm e}^{c t +x} + B\, {\mathrm e}^{-c t -x} -c$, which is not a travelling wave.
We consider ``weak forms" of travelling waves in the next section.

\medskip

\begin{remark}
The nonlocal symmetries of the CH equation studied in \cite{reyes:R4,G-R,HHR} do not transform into nonlocal 
symmetries of the fifth-order CH-type equation (\ref{ch410}): 
it is proven in \cite{reyes:R4} that the Camassa-Holm equation admits the nonlocal symmetry
$V = \gamma \exp(\delta/\lambda)\, \partial/\partial u$, in which $\gamma$ satisfies (\ref{reyes:tpart3}) and
$\delta$ is a potential of the conservation law (\ref{reyes:cl3}) but, on the other hand,
%Also, it is observed in \cite{reyes:keti}
%that a nonlocal symmetry of the same form as $V$ allows one to {\em classify} all integrable equations belonging
%to a one-parameter family of equations admitting quadratic pseudopotentials and conservation laws (see
%\cite[Theorem 6.6]{reyes:keti}).
 Equation (\ref{ch410}) does not admit a symmetry similar to $V$. Indeed,
computations carried out with the help of the MAPLE packages DifferentialGeometry and JetCalculus tell us that
it is not possible to choose $L \in \R$ so that $V = \gamma \exp(L\,\delta)\, \partial/\partial u$ --- in which
$\gamma$ solves (\ref{g1}), (\ref{g2}) and $\delta$ is a potential of the conservation law (\ref{g2}) ---
be a symmetry of (\ref{ch410}). More generally we can prove:

\begin{adjustwidth}{1cm}{}

\begin{proposition}
The fifth order CH-type equation $(\ref{ch410})$ does not admit a non-trivial nonlocal symmetry of the form
$V = f(\gamma ,\delta)\, \partial/\partial u$, in which $\gamma$ is a solution to $(\ref{g1})$ and $(\ref{g2})$,
and $\delta$ is a potential of the conservation law $(\ref{g2})$.
\end{proposition}
\begin{proof}
The method of proof is standard, and so we only sketch its main points. We use the MAPLE packages
DifferentialGeometry and JetCalculus in order to carry out our computations.

\medskip

Let $\Delta$ be the left hand side of Equation (\ref{ch410}). We consider a vector $V$ as in the enunciate of the
proposition and we compute the Lie derivative
$$
L_{pr (V)}\Delta\; ,
$$
in which $pr(V)$ is the fifth prolongation of $V$. This derivative depends on higher derivatives of $\gamma$ and
$\delta$. We get rid of these derivatives using the four compatible equations (\ref{g1}), (\ref{g2}),
$\delta_x = \gamma$, $\delta_t = \lambda (u_x - u_{xxx}) - \lambda \gamma -
(u - u_{xx})\gamma$, and their differential consequences. We obtain a long expression which depends only of
$x$-derivatives of $u$; in fact, the highest $x$-derivative that appears in this expression is $u_{xxxxxxx}$.
We will call this expression (and the ones obtained from it as explained below) $E$, simply. Differentiating $E$
with respect to $u_{xxxxxxx}$ and then differentiating the resulting
expression with respect to $u$, we obtain the necessary condition
$$ f_{\gamma \gamma} = 0 $$
for $V$ to be a symmetry, this is, $f(\gamma , \delta) = f_1(\delta) \gamma + f_2(\delta)$. Replacing into $E$,
differentiating with respect to $u_{xxxxxxx}$, and then differentiating the resulting expression with respect to
$u_x$, yield the new necessary condition $f_1(\delta) = C1\,\exp(3\delta/\lambda)$. Replacing
this constraint into $E$ and differentiating with respect to $u_{xxxxxxx}$ once again, we find that our next
necessary condition is $C1=0$. We replace this new constraint into $E$ and differentiate the resulting
expression with respect to $u_{xxxxxx}$ and to $u_x$. We obtain the new necessary condition $f_2(\delta) = C3 +
C4 \exp(2 \delta/\lambda)$. Replacing one last time into $E$ and differentiating with respect to $u_{xxxxxx}$,
we obtain the conditions $C3 = C4 =0$, so that if a vector field $V=f(\gamma ,\delta)\, \partial/\partial u$ were
a symmetry of (\ref{ch410}), then the function $f$ had to vanish identically.
\end{proof}
\end{adjustwidth}
\end{remark}
}

\section{Pseudo-peakons}

%In this section we study solutions to our fifth order CH-type equation (\ref{FOCHInt1}). As pointed out after
%Proposition 2, we can find explicit solutions rather easily. Besides the elementary solution already reported
%therein, we can check, for instance, that
%$$
%u(x,t) = \left( c_1\,{e^{-2\,x}}+ c_2\,{e^{-2\,x}} x + c_3 + c_4 \,x \right){e^{x-d(t) }} \; ,
%$$
%in which $c_1,\dots,c_4$ are constant numbers and $d(t)$ is an arbitrary function of $t$, solves (\ref{FOCHInt1}).
%This function $u(x,t)$ is not a solution to the standard Camassa-Holm equation. 

The main goal of this section is
to show that the integrable Equation (\ref{FOCHInt1}) admits pseudo-peakon and
multi-pseudo-peakon solutions, as anticipated in Section 1.
%to the CH-type equation
%(\ref{FOCHInt1}), which is the major goal in this section.

%\newpage
%\ColorFoil{\hfill{\small \textcolor{magenta}{-- Integrable
 % \ Hierarchy  }\\$ $\\}\normalsize{Pseudo-Peakons}}

Casting the regular travelling wave setting $\xi=x-ct$ in
the CH-type equation (\ref{FOCHInt1}), through a lengthy computation we obtain
the following single {\em pseudo-peakon} solution:
\begin{equation} \label{pp1}
u=\frac{c}{2} e^{-|\xi|}(1+|\xi|)\; , \quad \xi=x-ct\; ,
\end{equation}
which looks like a peakon since there are absolute-value functions
involved. But, this function in spirit  has %only
continuous derivatives up to the second order,
\begin{equation} \label{pp2}
u'=-\frac{c}{2} e^{-|\xi|} \xi\; , \quad \quad  u''=\frac{c}{2} e^{-|\xi|} (|\xi|-1)\; ,
\end{equation}
{\color{black} which show us that the solution $u$ is differentiable, % to the first
%order derivative,
with continuous and bounded second order derivative, but whose
%but no continuous
third order derivative blows up} (see Figures \ref{fig:FOCH} and \ref{fig:FOCH1}). %(see the graph plotted below).

\begin{figure}
\begin{minipage}[t]{0.5\linewidth}
\centering
\includegraphics[height=3cm,width=5cm]{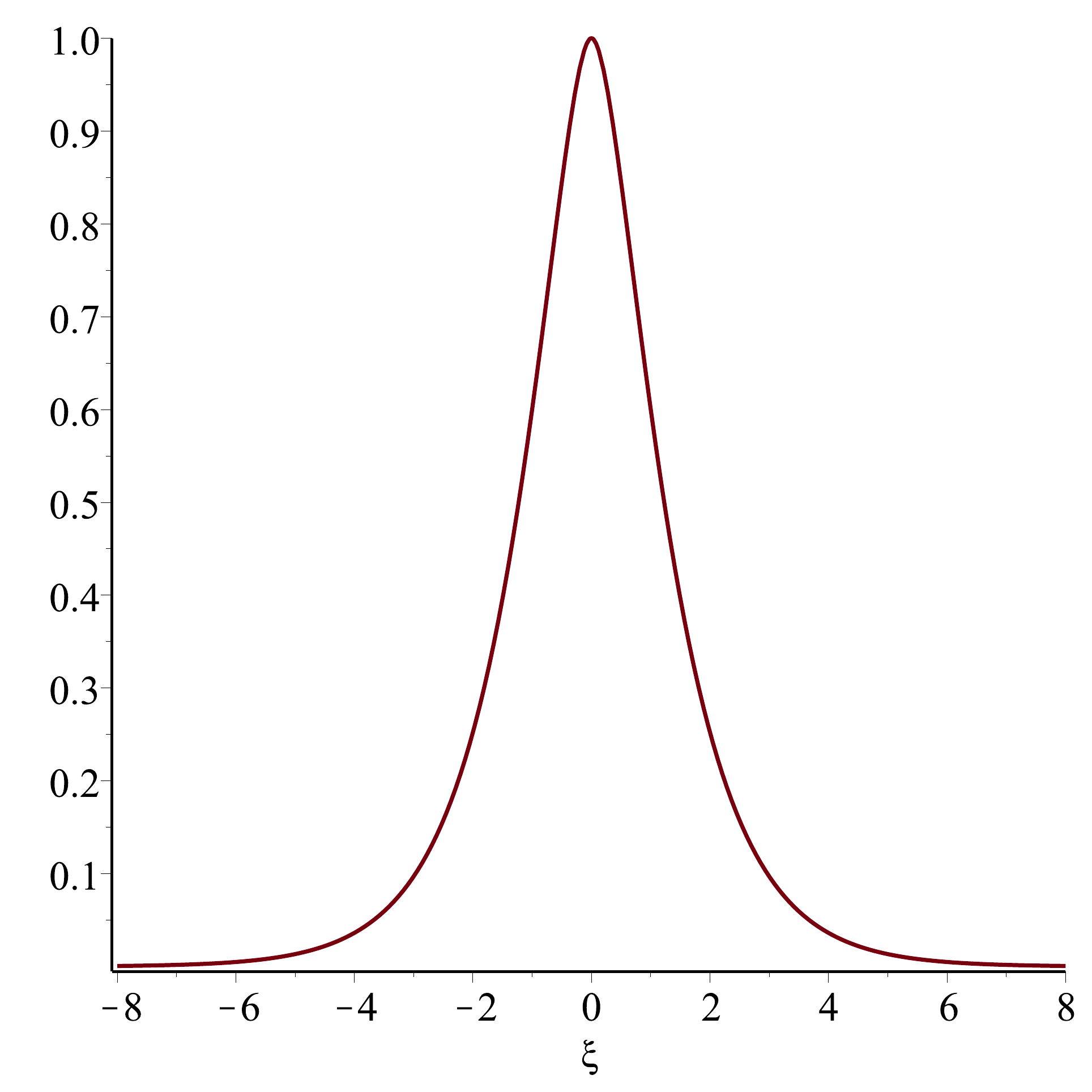}
\caption{ The single pseudo-peakon solution (\ref{pp1})}
\label{fig:FOCH}
\end{minipage}
\hspace{1.9ex}
\begin{minipage}[t]{0.5\linewidth}
\centering
\includegraphics[height=3.60cm,width=5.0cm]{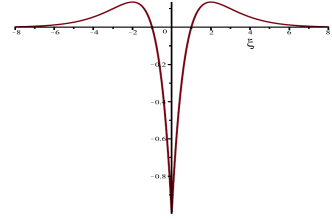}
\caption{ The peaked second derivative (\ref{pp2})}
\label{fig:FOCH1}
\end{minipage}
%\hspace{1.9ex}
%\begin{minipage}[t]{0.3\linewidth}
%\centering
%\includegraphics[width=2.3in]{Pseudo-peakon-3d-rev.JPG}
%\caption{\small{The two pseudo-peakon interaction plotted in 3D
%The two-peakon solution determined by (\ref{42pq2}). Solid line: $u(x,t)$; Dashed line: $v(x,t)$; Black: $t=-0.5$; Blue: $t=-1$.
%}}
%\label{F23}
%\end{minipage}
\end{figure}

%\newpage
%\ColorFoil{\hfill{\small \textcolor{magenta}{-- Integrable
 % \ Hierarchy  }\\$ $\\}\normalsize{$N$-Pseudo-Peakons}}

\medskip

We can also compute multi-pseudo-peakon solutions. They are of the form
\begin{equation} \label{mpp}
 u=\sum_{j=1}^N \frac{p_j(t)}{2} e^{-|x-q_j(t)|}(1+|x-q_j(t)|)\; ,
\end{equation}
 where $p_j(t), \ q_j(t) $ satisfy the following canonical Hamiltonian dynamical system
\beq
\dot{q}_j&=& \frac{\pa H}{\pa p_j}, \label{hs1} \\
\dot{p}_j&=&- \frac{\pa H}{\pa q_j},
\eeq
with Hamiltonian
function:\textcolor{black}{ \beq H&=&
\frac{1}{2}\sum_{i,j=1}^Np_ip_je^{-|q_i-q_j|}\; . \label{hs3} \eeq}

A crucial observation is that (\ref{hs1})--(\ref{hs3}) coincides exactly with the finite-dimensional
peakon dynamical system of the CH equation, see \cite{reyes:CH}. As we will explain momentarily, this fact allows 
us to have a full picture of multi-pseudo-peakon solutions. 
First, let us calculate explicitly $2$-pseudo-peakons. When $N=2$, we have the $2-$pseudo-peakon equations below:
\begin{eqnarray}
\left\{\begin{array}{l}
p_{1,t}= p_1p_2\, {\rm sgn}(q_1-q_2)\, e^{-|q_1-q_2|}\; ,\\
p_{2,t}= p_1p_2\, {\rm sgn}(q_2-q_1)\, e^{-|q_1-q_2|}\; ,\\
q_{1,t}= p_1+p_2\, e^{-|q_1-q_2|}\; ,\\
q_{2,t}= p_2+p_1\, e^{-|q_1-q_2|}\; ,
\end{array}\right.
\label{2-Pseudo-Peakon}
\end{eqnarray}
which can be solved with the following explicit solutions:
\begin{eqnarray}
\left\{\begin{array}{l}
p_1(t)=-p_2(t)=A \coth (At),\\% {\rm sgn}(q_1-q_2) \exp (-|q_1-q_2|),\\
q_1(t)=-q_2(t)=\ln \cosh (At),
\end{array}\right.
\label{2-Pseudo-Peakon-pq}
\end{eqnarray}
where $A$ is an arbitrary constant.
Thus, the 2-pseudo-peakon solution of the fifth order CH-type equation (\ref{FOCHInt1}) is given by
%the following explicit formula:
{\small \begin{eqnarray}
 u(x,t)&=&\frac{p_1(t)}{2} e^{-|x-q_1(t)|}(1+|x-q_1(t)|)+\frac{p_2(t)}{2} e^{-|x-q_2(t)|}(1+|x-q_2(t)|)\nonumber\\
 &=&\frac{A}{2}\coth (At)\left[ e^{-|x-\ln \cosh (At)|}(1+|x-\ln \cosh (At)|)-e^{-|x+\ln \cosh (At)|}(1+|x+
    \ln \cosh (At)|)  \right], \nonumber\\
 & & \label{2-Pseudo-Peakon-u}
 \end{eqnarray}
 }
where $ \cosh (At)=\frac{e^{At}+e^{-At}}{2}$, and $ \coth (At)=\frac{e^{At}+e^{-At}}{e^{At}-e^{-At}}$.
If we fix time $t=\frac{\cosh^{-1}e}{A}$ and we select $A=\frac{2}{\coth(\cosh^{-1} e)}$, the above
2-pseudo-peakon solution reads as the following simplest form
{\small \begin{eqnarray*}
 u(x,t) &=&e^{-|x-1|}(1+|x-1|)-e^{-|x+1|}(1+|x+1|),
 \end{eqnarray*}}
which we may plot in a 2D picture for the two pseudo-peakon interaction (see Figure \ref{F21}).

\begin{figure}
\begin{minipage}[t]{0.5\linewidth}
\centering
\includegraphics[width=2.0in]{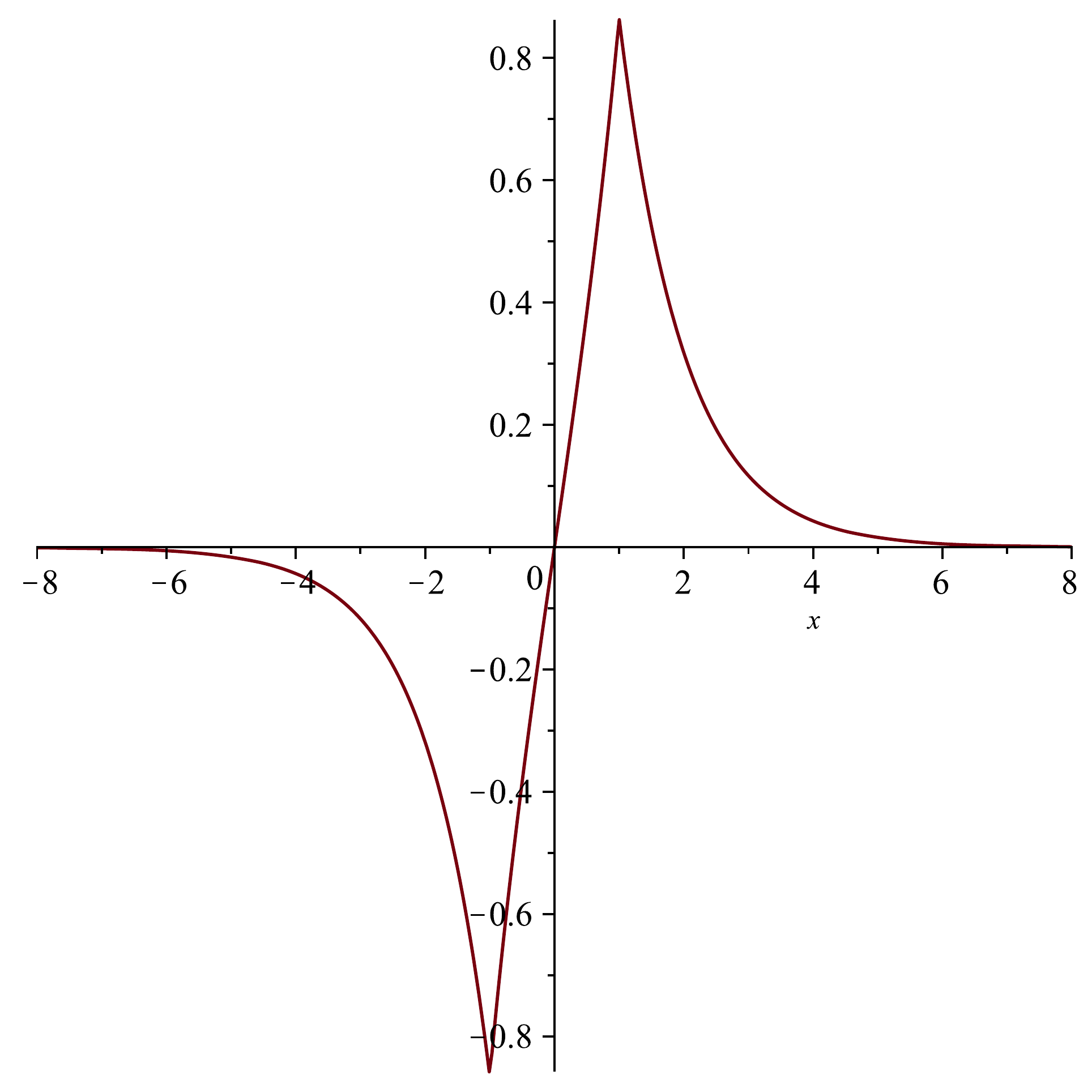}
\caption{\small{The two pseudo-peakon interaction plotted in 3D for negative time $t$. %??? %The M-shape peakon solution given by (\ref{suv21c1}). Solid line: $u(x,t)$; Dashed line: $v(x,t)$; Black: $t=0$; Blue: $t=-1$.
}}
\label{F21}
\end{minipage}
\hspace{2.9ex}
\begin{minipage}[t]{0.5\linewidth}
\centering
\includegraphics[width=3.3in]{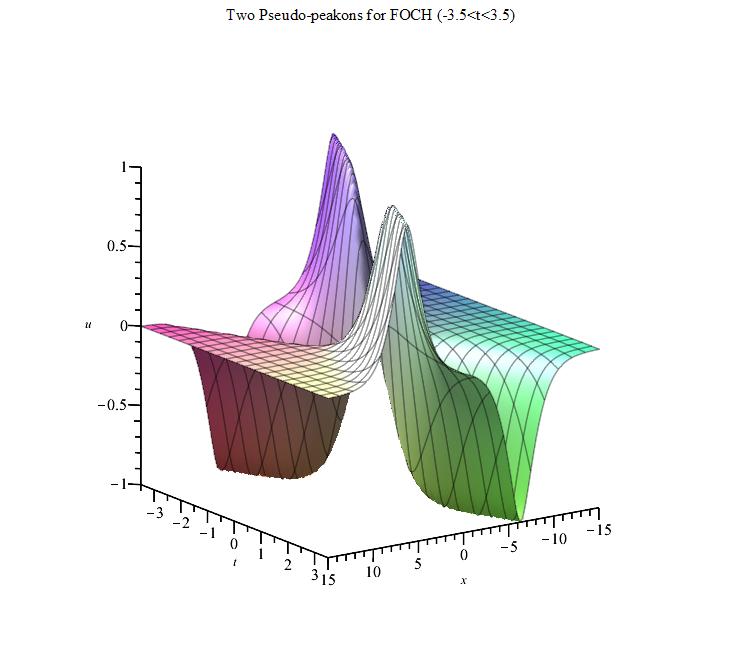}
\caption{\small{The two pseudo-peakon interaction plotted in 3D for all times.
%The two-peakon solution determined by (\ref{42pq2}). Solid line: $u(x,t)$; Dashed line: $v(x,t)$; Black: $t=-0.5$; Blue: $t=-1$.
}}
\label{F31}
\end{minipage}
\end{figure}

%\begin{figure}%[H]
%\centering
%\includegraphics[height=6cm,width=10cm]{2-Pseudo-peakon.pdf}
%\caption{ The two pseudo-peakon interaction}.
%\label{fig:FOCH2}
%\end{figure}

%\begin{center}
%\resizebox{3.20in}{!}{\includegraphics{Pseudo-peakon-3d-Ngtime.pdf}}
%\resizebox{3.20in}{!}{\includegraphics{Pseudo-peakon-3d-PS-time.pdf}}
%\end{center}

%\begin{figure}%[H]
%\centering
%\includegraphics[height=6.5cm,width=12cm]{Pseudo-peakon-3d.pdf}
%\caption{ The two pseudo-peakon interaction plotted in 3D}.
%\label{fig:FOCH2}
%\end{figure}
\begin{figure}
\begin{minipage}[t]{0.5\linewidth}
\centering
\includegraphics[width=3.3in]{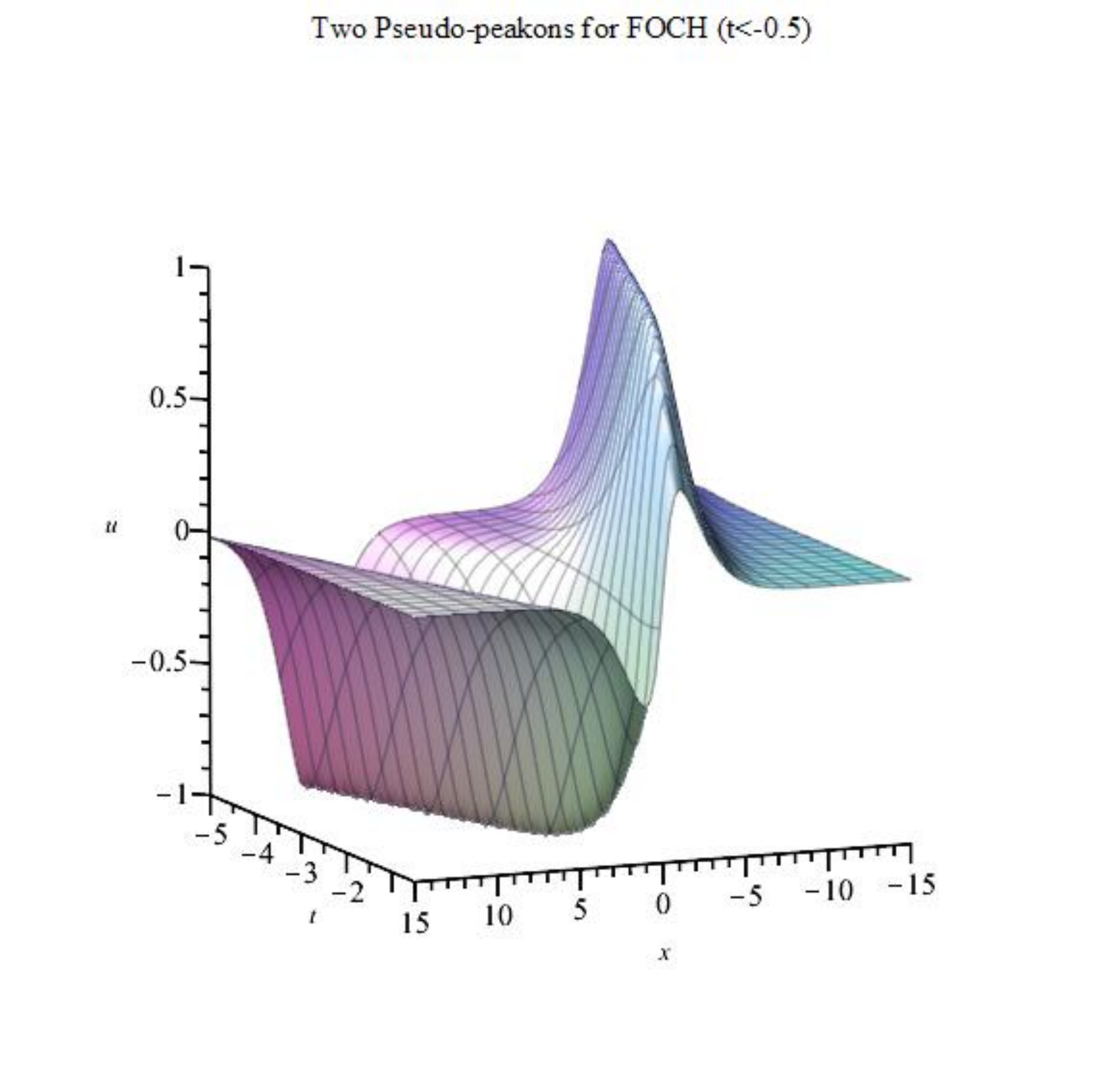}
\caption{\small{The two pseudo-peakon interaction plotted in 3D for negative time $t$. %??? %The M-shape peakon solution given by (\ref{suv21c1}). Solid line: $u(x,t)$; Dashed line: $v(x,t)$; Black: $t=0$; Blue: $t=-1$.
}}
\label{F3N}
\end{minipage}
\hspace{2.9ex}
\begin{minipage}[t]{0.5\linewidth}
\centering
\includegraphics[width=3.3in]{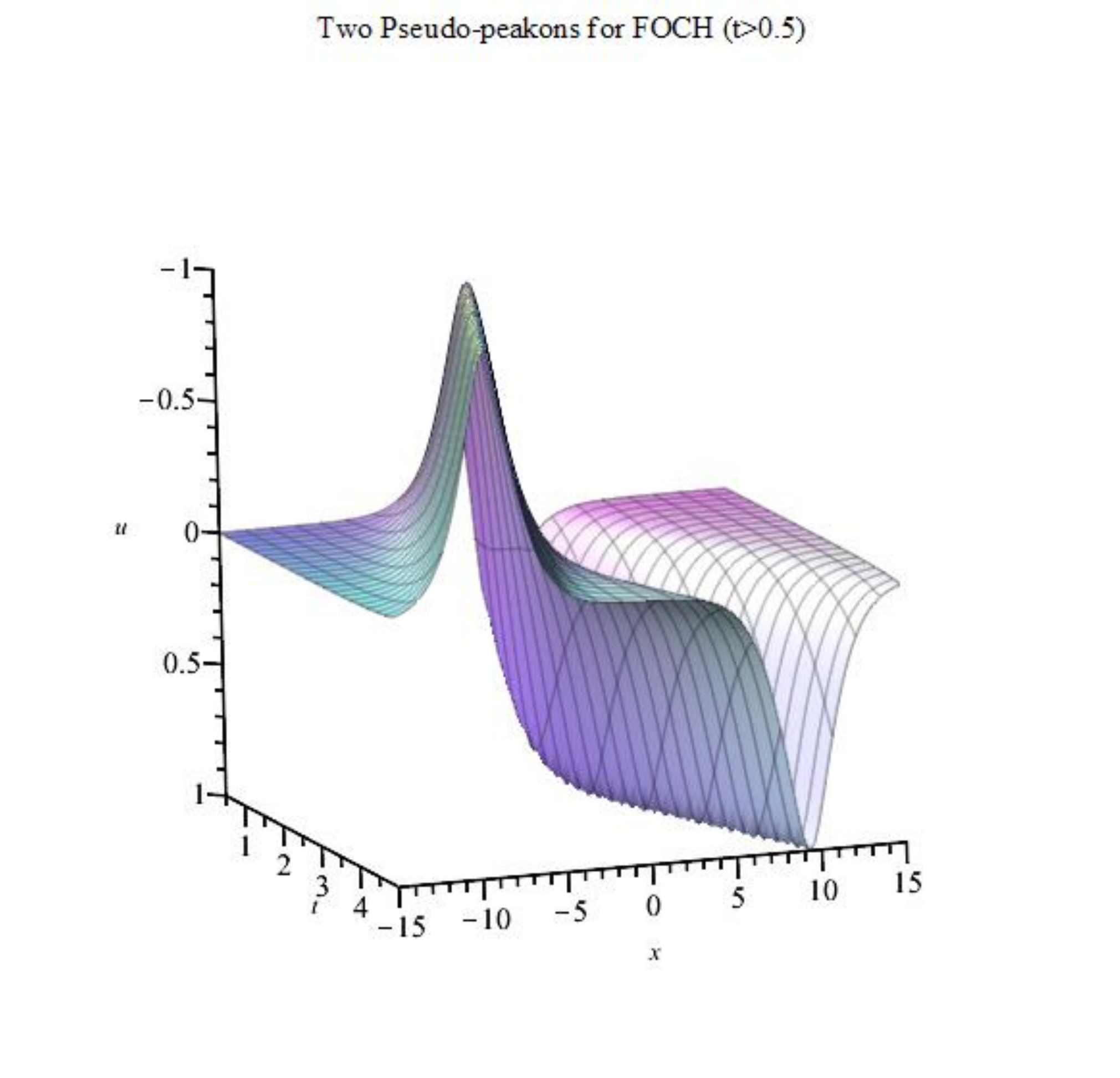}
\caption{\small{The two pseudo-peakon interaction plotted in 3D for positive time $t$. %??? %The two-peakon solution given by (\ref{suv21c2}). Solid line: $u(x,t)$; Dashed line: $v(x,t)$; Black: $t=0$; Blue: $t=-1$.
}}
\label{F3P}
\end{minipage}
%\hspace{1.9ex}
%\begin{minipage}[t]{0.3\linewidth}
%\centering
%\includegraphics[width=2.3in]{Pseudo-peakon-3d-rev.JPG}
%\caption{\small{The two pseudo-peakon interaction plotted in 3D
%The two-peakon solution determined by (\ref{42pq2}). Solid line: $u(x,t)$; Dashed line: $v(x,t)$; Black: $t=-0.5$; Blue: $t=-1$.
%}}
%\label{F23}
%\end{minipage}
\end{figure}

\begin{remark}
The plots below show $2$-pseudo-peakons in 3D.
Figure \ref{F31} shows a 3D interactional dynamics of the $2$-pseudo-peakon solution for all times. Figure \ref{F3N} and Figure \ref{F3P} show a 3D interactional dynamics of the two-pseudo-peakon solution for the negative times and the positive times, respectively. During the interaction of two-pseudo-peakons, it follows from the explicit solution (\ref{2-Pseudo-Peakon-u}) that the solution $u$ suddenly crashes to zero when the time $t$ passes from negative to positive via $t=0$. After the time $t=0$, the two-pseudo-peakon solution continues
travelling from left to right, but the amplitudes already flipped along with the time (see Figure \ref{F3N} and Figure \ref{F3P} for details).
\end{remark}

{\color{black} Now we consider multi-pseudo-peakons in full generality. It is known that the Hamiltonian system
(\ref{hs1})--(\ref{hs3})
%the multi-pseudo-peakon system
 ---that in our context will be called the multi-pseudo-peakon system--- is completely
integrable in the Liouville sense. This fact is discussed in \cite{reyes:CH} and fully studied by
 Calogero and Fran\c{c}oise in \cite{CF}. Since the Hamiltonian (\ref{hs3}) is not continuously differentiable,
 we cannot conclude that the trajectories of the system are given by quadratures via the Arnold-Liouville theorem.
 {\em However}, in the papers \cite{reyes:bss,reyes:bss1} Beals, Sattinger and Szmigielski are able to
solve the Hamiltonian system (\ref{hs1})--(\ref{hs3}) via inverse spectral methods and continued fractions. More
precisely we have, after the summary appearing in \cite[Theorem 2.1]{CCH}:

\begin{theorem}
The solutions of the Hamiltonian system
\begin{equation} \label{hsj}
\frac{d x_j}{d\tau} = \frac{\partial H}{\partial m_j}\; , \quad \; \frac{d m_j}{d\tau} = -
\frac{\partial H}{\partial x_j} \; , \quad \; H(x_1, \cdots,x_N,m_1,\cdots , m_N) = \frac{1}{4}
\sum_{j,k=1}^N m_j m_k e^{-2|x_j - x_k|}
\end{equation}
are given by
\begin{equation} \label{hss}
x_j = \frac{1}{2}\log \left( \frac{1+y_j}{1-y_j} \right) \; , \quad \quad \quad m_j = g_j (1-y_j^2)\; ,
\end{equation}
in which
$$
y_j = 1 - \frac{\Delta_{N-j}^2}{\Delta_{N-j+1}^0}\; , \quad \quad \quad
g_j = \frac{(\Delta_{N-j+1}^0)^2}{\Delta_{N-j+1}^1\,\Delta_{N-j}^1} \; .
$$
The functions $\Delta_k^l(\tau)$ are given by
$$
\Delta_k^l = \det( A_{i+j+l} (\tau) )_{i,j=1}^{k-1}\; , \quad \quad
A_k(\tau) = \sum_{j=0}^N (-\lambda_j)^k\,a_j(\tau)\; , \quad a_0(\tau) = \frac{1}{2}\; , \quad \lambda_0 = 0\; ,
$$
with $\frac{d}{d\tau} a_j(\tau) = - \frac{2 a_j(\tau)}{\lambda_j}$ and $\lambda_j \neq 0$ for $j \geq 1$.
\end{theorem}
\noindent In this theorem we assume that the numbers $\lambda_j$ are all distinct and have the same sign, and that
the initial conditions $a_j(0)$ are positive, see \cite[p. 158]{CCH}.

Now, in order to obtain solutions to the system (\ref{hs1})--(\ref{hs3}), we apply the symplectic transformation
$x_j \mapsto (1/2)\, q_j$, $m_j \mapsto 2 p_j$. Then, the Hamiltonian appearing in  (\ref{hsj}) becomes
(\ref{hs3}), and trajectories of (\ref{hsj}) map onto trajectories of (\ref{hs1})--(\ref{hs3}). We adjust the time
evolution by setting $t = 2 \tau$ and we are ready: replacing $q_j(t)$ and $p_j(t)$ into our formula (\ref{mpp})
we obtain explicit expressions for $N$-pseudo-peakons.
}

\section{The Cauchy problem for the CH-type equation}
In this section we consider the Cauchy problem for the fifth order CH-type equation \eqref{FOCHInt1}. We rewrite
\eqref{FOCHInt1} as
\begin{eqnarray} \label{CHtype}
\left\{
  \begin{array}{ll}
m_t+2 \,(u-u_{xx})_x\, m + (u-u_{xx})\, m_x = 0,\\
m = -A_4(u)=(1-\partial_x^2)^2u=u-2u_{xx}+u_{xxxx}\; .
  \end{array}
\right.\end{eqnarray}

The operator $(1-\partial_x^2)^{-2}$ can be expressed by
$$(1-\partial_x^2)^{-2}f=G*f=\int_{\R}G(x-y)f(y)dy$$
for any $f\in L^2(\R)$ with $G=\frac{1}{4}(1+|x|)e^{-|x|}$.
It follows
\begin{eqnarray}\label{u}
u(x,t)&=&G*m=\frac{1}{4}\int_{\mathbb{R}}(1+|x-\xi|)e^{-|x-\xi|}m(\xi,t)d\xi\nonumber\\
&=&\frac{1}{4}e^{-x}\int_{-\infty}^x(1+x-\xi)e^{\xi}m(\xi,t)d\xi+\frac{1}{4}e^{x}\int_x^{+\infty}(1-x+\xi)e^{-\xi}m(\xi,t)d\xi.\nonumber\\
\end{eqnarray}
Then,
\begin{eqnarray*}
u_x(x,t)&=&-\frac{1}{4}e^{-x}\int_{-\infty}^x(1+x-\xi)e^{\xi}m(\xi,t)d\xi+\frac{1}{4}e^{x}\int_x^{+\infty}(1-x+\xi)e^{-\xi}m(\xi,t)d\xi\nonumber\\
&&+\frac{1}{4}e^{-x}\int_{-\infty}^xe^{\xi}m(\xi,t)d\xi-\frac{1}{4}e^{x}\int_x^{+\infty}e^{-\xi}m(\xi,t)d\xi
\end{eqnarray*}
and
\begin{eqnarray}\label{uxx}
u_{xx}(x,t)&=&\frac{1}{4}e^{-x}\int_{-\infty}^x(1+x-\xi)e^{\xi}m(\xi,t)d\xi+\frac{1}{4}e^{x}\int_x^{+\infty}(1-x+\xi)e^{-\xi}m(\xi,t)d\xi\nonumber\\
&&-\frac{1}{2}e^{-x}\int_{-\infty}^xe^{\xi}m(\xi,t)d\xi-\frac{1}{2}e^{x}\int_x^{+\infty}e^{-\xi}m(\xi,t)d\xi.
\end{eqnarray}
\eqref{u} minus \eqref{uxx}, we have
\begin{eqnarray}\label{u-uxx}
(u-u_{xx})(x,t)=\frac{1}{2}e^{-x}\int_{-\infty}^xe^{\xi}m(\xi,t)d\xi+\frac{1}{2}e^{x}\int_x^{+\infty}e^{-\xi}m(\xi,t)d\xi.
\end{eqnarray}
Differentiating $u-u_{xx}$ with respect to $x$, we have
\begin{eqnarray}\label{ux-uxxx}
(u_x-u_{xxx})(x,t)=-\frac{1}{2}e^{-x}\int_{-\infty}^xe^{\xi}m(\xi,t)d\xi+\frac{1}{2}e^{x}\int_x^{+\infty}e^{-\xi}m(\xi,t)d\xi.
\end{eqnarray}

\subsection{Local well-posedness and blow-up scenario}
Firstly, we present the local well-posedness theorem for the CH type equation \eqref{CHtype}.
\begin{theorem}\label{Local}
Let $u_0\in H^s(\mathbb{R})$ with $s>\frac{7}{2}$. Then there exist a $T>0$ depending on $\|u_0\|_{H^s}$, such
that the CH type equation \eqref{CHtype} has a unique solution
$$u\in C([0,T);H^s(\mathbb{R}))\cap C^1([0,T);H^{s-1}(\mathbb{R})).$$
Morever, the map $u_0\in H^s\rightarrow u\in C([0,T);H^s(\mathbb{R}))\cap C^1([0,T);H^{s-1}(\mathbb{R}))$ is
continuous.
\end{theorem}
The proof (via Kato's theory) is similar to the ones appearing in \cite{CE} and
\cite[Section 3]{Guill}, and therefore we have omitted the details.  
The maximum value of $T$ appearing in Theorem \ref{Local} is called the lifespan of the solution. If
 $T<\infty$, that is, $$\lim_{t\to T^-}\|u\|_{H^s}=\infty\; ,$$ we say the solution blows up in finite time.  
Now we can use one of the conserved densities found in Section 2, namely, the Hamiltonian (\ref{cd8}). We have 
\begin{eqnarray}
\int_{\R} (u^2+3u_x^2+3u_{xx}^2+u_{xxx}^2) dx = \int_{\R}(u_0^2+3u_{0x}^2+3u_{0xx}^2+u_{0xxx}^2) dx := E_0^2\; .
\end{eqnarray}
We note that $$\|u_{x}\|_{L^\infty}\leq \frac{\|u_{x}\|_{H^1}}{\sqrt{2}}\leq \frac{E_0}{\sqrt{2}}\; .$$ 
This means that the first-order derivative of $u$ is always bounded. Let us present the precise blow-up scenario.

\begin{theorem}\label{blow-up scenario}
 Assume that $u_0\in H^4(\mathbb{R})$ and let $T$ be the maximal existence time of the solution $u(x,t)$ to the CH type equation \eqref{CHtype} with the initial data $u_0(x)$, then the  corresponding solution of the CH type equation \eqref{CHtype} blows up in finite time if and only if
{\color{black}\begin{eqnarray*}
\lim_{t\to T}\sup_{x\in{\mathbb{R}}}\{(u_{xxx})(x,t)\}=\infty.
 \end{eqnarray*}}
\end{theorem}
\begin{remark}
This theorem means that the blow-up mechanism of (\ref{CHtype}) is not the same as in the CH case: for the CH equation, see \cite{CE,CE1}, blow-up phenomenon means that the first-order derivative tends to negative infinity. However, for our equation blow-up means that the third-order derivative tends to infinity.
\end{remark}
\begin{proof}
By direct calculation, we have
$$\|u\|_{H^4}^2\leq\|m\|_{L^2}^2\leq 3\|u\|_{H^4}^2\; .$$
Multiplying \eqref{CHtype} by $m$, we obtain by straighforward computations,
\begin{eqnarray*}
\frac{d}{dt}\int_\mathbb{R}m^2dx&=&-2\int_\mathbb{R}2\,(-u_{xx} + u)_x\, m^2 + (-u_{xx} + u)\, m_x\,mdx\\
&\leq&-3\inf_{x\in{\mathbb{R}}}\{u_x-u_{xxx}\}\int_\mathbb{R}m^2dx\\
&\leq&{\color{black} \left(\frac{3}{\sqrt{2}}E_0+3\sup_{x\in{\mathbb{R}}}\{u_{xxx}\}\right)\int_\mathbb{R}m^2dx.}
\end{eqnarray*}
If $$\sup_{x\in{\mathbb{R}}}\{u_{xxx}\}\leq M,$$ then
\begin{eqnarray*}
\frac{d}{dt}\int_\mathbb{R}m^2dx\leq{\color{black}\left(\frac{3}{\sqrt{2}}E_0+3M\right)}\int_\mathbb{R}m^2dx.
\end{eqnarray*}
By using the Gronwall inequality,
\begin{eqnarray*}
\|m\|_{L^2}^2\leq e^{{\color{black}\left(\frac{3}{\sqrt{2}}E_0+3M\right)}}\|m_0\|_{L^2}^2.
\end{eqnarray*}
Therefore the $H^4$-norm of the solution is bounded on $[0,T)$. On the other hand, by the Sobolev's embedding
{\color{black}$\|u_{xxx}\|_{L^\infty}\leq\|u\|_{H^4}$}. This inequality tells us that if $H^4$-norm of the solution is
bounded, then the $L^\infty$-norm of {\color{black}$u_{xxx}$} is bounded. We have completed the proof of Theorem
\ref{blow-up scenario}.
\end{proof}
\subsection{Global existence and blow up phenomena}
In this subsection, firstly, we establish a sufficient condition that
guarantees global existence of the solution to our CH type equation \eqref{CHtype}. We give the particle
trajectory as
\begin{eqnarray}\label{PT}
\left\{
  \begin{array}{ll}
   q_t(x,t)=(u-u_{xx})(q(x,t),t), \qquad 0<t<T, x\in \mathbb{R},\\
   q(x,0)=x, \qquad x\in \mathbb{R},
  \end{array}
\right.
\end{eqnarray}
where $T$ is the lifespan of the solution.
Taking derivative of \eqref{PT} with respect to $x$, we obtain
$$\frac{dq_t(x,t)}{dx}=q_{tx}=((u_x-u_{xxx})q_x)(q(x,t),t), \qquad t\in(0,T).$$
Therefore
\begin{eqnarray*}
\left\{
  \begin{array}{ll}
   q_x=exp\{\int^t_0(u_x-u_{xxx})(q,s)ds\}, \quad 0<t<T,\quad x\in \mathbb{R},\\
   q_x(x,0)=1, \qquad x\in \mathbb{R},
  \end{array}
\right.
\end{eqnarray*}
which is always positive before the blow-up time. Therefore, the function $q(x,t)$ is an increasing diffeomorphism
of the line before blow-up. In fact, direct calculation yields
$$\frac{d}{dt}(m(q(x,t),t)q^{b}_x)=[m_t(q)+(u_x-u_{xxx})(q,t)m_x(q)+2(u_x-u_{xxx})(q,t)m(q)]q^{2}_x=0.$$
Hence, the following identity can be proved:
\begin{eqnarray}\label{mq}
m(q(x,t),t)q^2_x=m_0(x),\qquad 0<t<T, x\in \mathbb{R}.
\end{eqnarray}
From \eqref{mq}, we know that if the initial data $m_0(x,t)\geq0$, then $m(q(x_0,t),t)\geq0$.
Before stating our main results, we recall once again the useful conservation law 
\begin{eqnarray}\label{CL1}
\int_{\R} (u^2+3u_x^2+3u_{xx}^2+u_{xxx}^2) dx = \int_{\R}(u_0^2+3u_{0x}^2+3u_{0xx}^2+u_{0xxx}^2) dx := E_0^2
\end{eqnarray}
found in Section 2.

\begin{theorem}\label{GE1}
Suppose that $u_0\in H^4(\R)$, and
$m_0=(1-\partial_x^2)^2u_0$ does not change sign. Then the
corresponding solution to the CH-type equation \eqref{CHtype} exists globally.
\end{theorem}
\begin{proof}
From \eqref{mq}, we known that $m(x,t)$ does not change sign either.
By Theorem \ref{blow-up scenario}, we only need to bound $u_{xxx}(x,t)$ from above. Now, we already 
observed that $u_x$ is always bounded thanks to the conservation law \eqref{CL1}. It follows that if we can bound  
$u_x-u_{xxx}$ from below, we will have our proof.
 
For $m_0\geq0$, by \eqref{u-uxx} and \eqref{ux-uxxx}, we obtain
\begin{eqnarray*}
(u_x-u_{xxx})(x,t)+(u-u_{xx})(x,t)=e^{x}\int_x^{+\infty}e^{-\xi}m(\xi,t)d\xi\geq0.
\end{eqnarray*}
It follows
\begin{eqnarray*}
&&(u_x-u_{xxx})(x,t)\geq-(u-u_{xx})(x,t)\\
&&\geq -\left(\int_{\R}u^2+3u_{x}^2+3u_{xx}^2+u_{xxx}^2dx\right)^\frac{1}{2}=-E_0.
\end{eqnarray*}
Similarly, for $m_0\leq0$, we have
\begin{eqnarray*}
(u_x-u_{xxx})(x,t)\geq(u-u_{xx})(x,t)\geq-E_0.
\end{eqnarray*}
The proof of Theorem \ref{GE1} is completed.
\end{proof}

\begin{theorem}\label{GE2}
Suppose that $u_0\in H^4(\R)$, and that 
there exists $x_0\in \R$ such that $m_0(x)\leq 0$ on $(-\infty,x_0]$ and $m_0(x)\geq 0$ on $[x_0,\infty)$. Then
the corresponding solution to the CH-type equation \eqref{CHtype} exists globally.
\end{theorem}
\begin{proof}
Similarly to the argument in the proof of Theorem \ref{GE1}, we need to bound $u_x-u_{xxx}$ from below. From 
\eqref{mq}, we known that $m(x,t)\leq0$ on $(-\infty,q(x_0,t)]$ and $m(x,t)\geq0$ on $[q(x_0,t),\infty)$.
For the points  $x\in(-\infty,q(x_0,t)]$, we have
\begin{eqnarray*}
(u_x-u_{xxx})(x,t)=(u-u_{xx})(x,t)-e^{-x}\int_{-\infty}^{x}e^{\xi}m(\xi,t)d\xi\geq(u-u_{xx})(x,t).
\end{eqnarray*}
For the points $x\in[q(x_0,t),\infty)$, we have
\begin{eqnarray*}
(u_x-u_{xxx})(x,t)=-(u-u_{xx})(x,t)+e^{x}\int_x^{+\infty}e^{-\xi}m(\xi,t)d\xi\geq-(u-u_{xx})(x,t).
\end{eqnarray*}
This means that for any $x\in\R$ we have
\begin{eqnarray*}
(u_x-u_{xxx})(x,t)\geq-\|u-u_{xx}\|_{L^\infty}\geq-E_0,
\end{eqnarray*}
where we have used the conservation law \eqref{CL1} once again. We have 
completed the proof of Theorem \ref{GE2}.
\end{proof}

\begin{theorem}\label{BC1}
 Assume that $u_0\in H^4(\mathbb{R})$ and that there exists $x_0\in \R$ such that
\begin{eqnarray}\label{ibc}
(u_{0x}-u_{0xxx})(x_0)<-\frac{E_0}{\sqrt{2}}\; ,
 \end{eqnarray}
where $E_0$ is defined in \eqref{CL1}. Then, the corresponding solution $u(x,t)$ to CH type equation \eqref{CHtype} 
blows up at a finite time $T$ bounded by
\begin{eqnarray*}
T\leq\frac{1}{-\frac{1}{2}(u_{0x}-u_{0xxx})(x_0)+\frac{E_0^2}{2(u_{0x}-u_{0xxx})(x_0)}}\; .
 \end{eqnarray*}
\end{theorem}
\begin{remark}
Due to the fact that $u_{0x}$ is bounded by the conservation law $E_0$ with %as
$\|u_{0x}\|_{L^\infty}\leq \frac{\|u_{0x}\|_{H^1}}{\sqrt{2}}\leq \frac{E_0}{\sqrt{2}}$, %. So
the condition \eqref{ibc} holds true %can been
only if %added on $u_{0xxx}$ as
$u_{0xxx}(x_0)>\sqrt{2}E_0$.
\end{remark}
\begin{proof}
Let
\begin{eqnarray*}\label{It}
I(t)=\frac{1}{2}e^{-q(x_0,t)}\int_{-\infty}^{q(x_0,t)}e^{x}m(x,t)dx
\end{eqnarray*}
and
\begin{eqnarray*}\label{IIt}
II(t)=\frac{1}{2}e^{q(x_0,t)}\int_{q(x_0,t)}^{+\infty}e^{-x}m(x,t)dx.
\end{eqnarray*}
From \eqref{u-uxx} and \eqref{ux-uxxx}, we have
\begin{eqnarray*}
(u-u_{xx})(q(x_0,t),t)=I(t)+II(t)
\end{eqnarray*}
and
\begin{eqnarray*}
(u_x-u_{xxx})(q(x_0,t),t)=-I(t)+II(t).
\end{eqnarray*}
Differentiating $(u_x-u_{xxx})(q(x_0,t),t)$ with respect to $t$,
\begin{eqnarray}\label{dt}
&&\frac{d}{dt}(u_x-u_{xxx})(q(x_0,t),t)=-\frac{d}{dt}I(t)+\frac{d}{dt}II(t).
\end{eqnarray}
Then, we estimate $\frac{d}{dt}I(t)$.
\begin{eqnarray*}
\frac{d}{dt}I(t)&=&\frac{1}{2}(u-u_{xx})m(q(x_0,t),t)-\frac{1}{2}(u-u_{xx})e^{-q(x_0,t)}\int_{-\infty}^{q(x_0,t)}e^{x}m(x,t)dx\nonumber\\
&&+\frac{1}{2}e^{-q(x_0,t)}\int_{-\infty}^{q(x_0,t)}e^{x}m_t(x,t)dx.
\end{eqnarray*}
The third term in the right hand side can been estimated as
\begin{eqnarray*}
&&\frac{1}{2}e^{-q(x_0,t)}\int_{-\infty}^{q(x_0,t)}e^{x}m_t(x,t)dx\nonumber\\
&=&-\frac{1}{2}e^{-q(x_0,t)}\int_{-\infty}^{q(x_0,t)}e^{x}((u-u_{xx})m_x+2(u-u_{xx})_xm)(x,t)dx\nonumber\\
&=&-\frac{1}{2}e^{-q(x_0,t)}\int_{-\infty}^{q(x_0,t)}e^{x}(((u-u_{xx})m)_x+(u-u_{xx})_xm)(x,t)dx\nonumber\\
&=&-\frac{1}{2}(u-u_{xx})m(q(x_0,t),t)+\frac{1}{2}e^{-q(x_0,t)}\int_{-\infty}^{q(x_0,t)}e^{x}(((u-u_{xx})m)-(u-u_{xx})_xm)(x,t)dx\nonumber\\
&=&-\frac{1}{2}(u-u_{xx})m(q(x_0,t),t)+\frac{1}{2}e^{-q(x_0,t)}\int_{-\infty}^{q(x_0,t)}e^{x}\left((u-u_{xx})^2-(u-u_{xx})(u-u_{xx})_{xx}\right.\nonumber\\
&&\left.-(u-u_{xx})_x(u-u_{xx})+(u-u_{xx})_x(u-u_{xx})_{xx}\right)dx\nonumber\\
&=&-\frac{1}{2}(u-u_{xx})m(q(x_0,t),t)-\frac{1}{2}(u-u_{xx})(u-u_{xx})_{x}+\frac{1}{4}(u-u_{xx})_{x}^2\nonumber\\
&&+\frac{1}{2}e^{-q(x_0,t)}\int_{-\infty}^{q(x_0,t)}e^{x}\left((u-u_{xx})^2+\frac{1}{2}(u-u_{xx})_{x}^2\right)dx.
\end{eqnarray*}
It follows that
\begin{eqnarray*}
\frac{d}{dt}I(t)&=&-\frac{1}{2}(u-u_{xx})e^{-q(x_0,t)}\int_{-\infty}^{q(x_0,t)}e^{x}m(x,t)dx
-\frac{1}{2}(u-u_{xx})(u-u_{xx})_{x}+\frac{1}{4}(u-u_{xx})_{x}^2\nonumber\\
&&+\frac{1}{2}e^{-q(x_0,t)}\int_{-\infty}^{q(x_0,t)}e^{x}\left((u-u_{xx})^2+\frac{1}{2}(u-u_{xx})_{x}^2\right)dx\nonumber\\
&=&-\frac{1}{2}(u-u_{xx})^2+\frac{1}{4}(u-u_{xx})_{x}^2+\frac{1}{2}e^{-q(x_0,t)}\int_{-\infty}^{q(x_0,t)}e^{x}\left((u-u_{xx})^2+\frac{1}{2}(u-u_{xx})_{x}^2\right)dx.
\end{eqnarray*}
Note that
\begin{eqnarray*}
&&\int_{-\infty}^{q(x_0,t)}e^{x}\left((u-u_{xx})^2+(u-u_{xx})_{x}^2\right)dx\\
&\geq& 2\int_{-\infty}^{q(x_0,t)}e^{x}(u-u_{xx})(u-u_{xx})_{x}dx\\
&\geq& e^{q(x_0,t)}(u-u_{xx})^2(q(x_0,t),t)-\int_{-\infty}^{q(x_0,t)}e^{x}(u-u_{xx})^2dx,
\end{eqnarray*}
which yields that
\begin{eqnarray*}
\int_{-\infty}^{q(x_0,t)}e^{x}\left((u-u_{xx})^2+\frac{1}{2}(u-u_{xx})_{x}^2\right)dx\geq \frac{1}{2}e^{q(x_0,t)}(u-u_{xx})^2(q(x_0,t),t).
\end{eqnarray*}
Therefore,
\begin{eqnarray}\label{dtI}
\frac{d}{dt}I(t)\geq -\frac{1}{4}(u-u_{xx})^2(q(x_0,t),t)+\frac{1}{4}(u-u_{xx})_{x}^2(q(x_0,t),t).
\end{eqnarray}
Similarity,
\begin{eqnarray*}
\frac{d}{dt}II(t)&=&-\frac{1}{2}(u-u_{xx})m(q(x_0,t),t)+\frac{1}{2}(u-u_{xx})e^{q(x_0,t)}\int_{q(x_0,t)}^{+\infty}e^{-x}m(x,t)dx\nonumber\\
&&+\frac{1}{2}e^{q(x_0,t)}\int_{q(x_0,t)}^{+\infty}e^{-x}m_t(x,t)dx.
\end{eqnarray*}
The third term in the right hand side can been estimated as
\begin{eqnarray*}
&&\frac{1}{2}e^{q(x_0,t)}\int_{q(x_0,t)}^{+\infty}e^{-x}m_t(x,t)dx\\
&=&\frac{1}{2}(u-u_{xx})m(q(x_0,t),t)-\frac{1}{2}(u-u_{xx})(u-u_{xx})_{x}-\frac{1}{4}(u-u_{xx})_{x}^2\nonumber\\
&&-\frac{1}{2}e^{q(x_0,t)}\int_{q(x_0,t)}^{\infty}e^{-x}\left((u-u_{xx})^2+\frac{1}{2}(u-u_{xx})_{x}^2\right)dx.
\end{eqnarray*}
It follows that
\begin{eqnarray*}
\frac{d}{dt}II(t)&=&\frac{1}{2}(u-u_{xx})e^{q(x_0,t)}\int_{q(x_0,t)}^{+\infty}e^{-x}m(x,t)dx\nonumber\\
&&-\frac{1}{2}(u-u_{xx})(u-u_{xx})_{x}-\frac{1}{4}(u-u_{xx})_{x}^2\nonumber\\
&&-\frac{1}{2}e^{q(x_0,t)}\int_{q(x_0,t)}^{\infty}e^{-x}\left((u-u_{xx})^2+\frac{1}{2}(u-u_{xx})_{x}^2\right)dx.
\end{eqnarray*}
By same argument, we have
\begin{eqnarray*}
\int_{q(x_0,t)}^{\infty}e^{-x}\left((u-u_{xx})^2+\frac{1}{2}(u-u_{xx})_{x}^2\right)dx\geq \frac{1}{2}e^{-q(x_0,t)}(u-u_{xx})^2(q(x_0,t),t).
\end{eqnarray*}
Therefore
\begin{eqnarray}\label{dtII}
\frac{d}{dt}II(t)\leq\frac{1}{4}(u-u_{xx})^2(q(x_0,t),t)-\frac{1}{4}(u-u_{xx})_{x}^2(q(x_0,t),t).
\end{eqnarray}
Inserting \eqref{dtI} and \eqref{dtII} into \eqref{dt}, we have
\begin{eqnarray}\label{ddt}
&&\frac{d}{dt}(u_x-u_{xxx})(q(x_0,t),t)\nonumber\\
&=&(u-u_{xx})^2(q(x_0,t),t)-\frac{1}{2}(u-u_{xx})_x^2(q(x_0,t),t)\nonumber\\
&&-\frac{1}{2}e^{-q(x_0,t)}\int_{-\infty}^{q(x_0,t)}e^{x}\left((u-u_{xx})^2+\frac{1}{2}(u-u_{xx})_{x}^2\right)dx
\nonumber\\
&&-\frac{1}{2}e^{q(x_0,t)}\int_{q(x_0,t)}^{\infty}e^{-x}\left((u-u_{xx})^2+\frac{1}{2}(u-u_{xx})_{x}^2\right)dx.
\end{eqnarray}
Combining the above estimates into \eqref{ddt}, we obtain
\begin{eqnarray}\label{dddt}
\frac{d}{dt}(u_x-u_{xxx})(q(x_0,t),t)\leq\frac{1}{2}(u-u_{xx})^2(q(x_0,t),t)-\frac{1}{2}(uu_{xx})_x^2(q(x_0,t),t).
\end{eqnarray}
By the fact $\|f\|_{L^\infty}^2<\frac{1}{2}\|f\|_{H^1}^2$, we have
\begin{eqnarray*}
\|u-u_{xx}\|_{L^\infty}^2<\frac{1}{2}\|u-u_{xx}\|_{H^1}^2=\frac{1}{2}E_0^2.
\end{eqnarray*}
Let $\varphi(t)=(u_x-u_{xxx})(q(x_0,t),t)$, we can rewrite \eqref{dddt} as
\begin{eqnarray*}
\varphi'(t)\leq-\frac{1}{2}\varphi^2(t)+\frac{1}{4}E_0^2.
\end{eqnarray*}
By using the hypothesis of the theorem and a standard
argument on Riccati type equations, we have that there exists a time $T$ such that
\begin{eqnarray*}
\lim_{t\rightarrow T}\varphi(t)=-\infty.
\end{eqnarray*}
By the boundedness of $u_x$, we have that $u_{xxx}$ goes to infinity. 
This completes the proof of Theorem \ref{BC1}. 
\end{proof}

\begin{theorem}\label{BC2}
 Assume that $u_0\in H^4(\mathbb{R})$ and that there exists $x_0\in \R$ such that
 $m_0(x_0)=(1-\partial_x^2)^2u_0(x_0)=0$,
\begin{eqnarray}\label{bc2}
\int_{-\infty}^{x_0}e^x m_0(x,t)dx>0\quad \mathrm{and}\quad \int_{x_0}^{\infty }e^{-x}m_0(x,t)dx<0,
\end{eqnarray}
then the corresponding solution $u(x,t)$ to CH type equation \eqref{CHtype} blows up in finite time.
\end{theorem}
\begin{proof}
Thanks to \eqref{mq}, we obtain $m(q(x_0,t),t)=0$ for all $t$ in its lifespan. Inequality \eqref{dddt} is also
correct in this proof. The initial condition \eqref{bc2} means that $(u_{0}-u_{0xx})_x(q(x_0,t),t)<0$ and
$(u_{0}-u_{0xx})_x^2(x_0,t)>(u_{0}-u_{0xx})^2(q(x_0,t),t)$. We claim that $(u_{0}-u_{0xx})_x(q(x_0,t),t)$ is
decreasing, $(u-u_{xx})_x^2(q(x_0,t),t)>(u-u_{xx})^2(q(x_0,t),t)$ for all $t\geq 0$. Suppose not, i.e. there
exists a $t_0$ such that $(u-u_{xx})_x^2(q(x_0,t),t)>(u-u_{xx})^2(q(x_0,t),t)$ on $[0,t)$ and
$(u-u_{xx})_x^2(q(x_0,t_0),t_0)\leq(u-u_{xx})^2(q(x_0,t_0),t_0)$.
Recall \eqref{dtI} and \eqref{dtII}, we get on $[0,t_0)$
\begin{eqnarray*}
\frac{d}{dt}I(t)\geq -\frac{1}{4}(u-u_{xx})^2(q(x_0,t),t)+\frac{1}{4}(u-u_{xx})_{x}^2(q(x_0,t),t)\geq 0
\end{eqnarray*}
and
\begin{eqnarray*}
\frac{d}{dt}II(t)\leq\frac{1}{4}(u-u_{xx})^2(q(x_0,t),t)-\frac{1}{4}(u-u_{xx})_{x}^2(q(x_0,t),t)\leq 0.
\end{eqnarray*}
It follows from the continuity property of ODEs that
\begin{eqnarray*}
(u-u_{xx})^2(q(x_0,t),t)-(u-u_{xx})_{x}^2(q(x_0,t),t)=4I(t)II(t)<4I(0)II(0)<0,
\end{eqnarray*}
for all $t>0$, this implies that $t_0$ can be extended to the infinity. This is a contradiction. So the claim is
true. By using \eqref{dtI} and \eqref{dtII} again, we get
\begin{eqnarray}\label{d1}
&&\frac{d}{dt}[(u-u_{xx})_x^2-(u-u_{xx})^2](q(x_0,t),t)\nonumber\\
&=&-\frac{d}{dt}\left\{\int_{-\infty}^{q(x_0,t)}e^{x}m(x,t)dx\times\int_{q(x_0,t)}^{+\infty}e^{-x}m(x,t)dx\right\}
\nonumber\\
&=&-4\frac{d}{dt}[I(t)II(t)]\nonumber\\
&=&-4\frac{d}{dt}I(t)\times II(t)-4\frac{d}{dt}II(t)\times I(t)\nonumber\\
&\geq&-[(u-u_{xx})_x^2-(u-u_{xx})^2](q(x_0,t),t)II(t)+[(u-u_{xx})_x^2-(u-u_{xx})^2](q(x_0,t),t)I(t)\nonumber\\
&=&-(u-u_{xx})_x(q(x_0,t),t)[(u-u_{xx})_x^2-(u-u_{xx})^2](q(x_0,t),t),
\end{eqnarray}
where we have used $(u-u_{xx})_x(q(x_0,t),t)=-I(t)+II(t)$.
Recall \eqref{dddt}, it follows
\begin{eqnarray}\label{d2}
&&(u_x-u_{xxx})(q(x_0,t),t)\nonumber\\
&\leq& \int_0^t\frac{1}{2}[(u-u_{xx})^2(q(x_0,s),s)-(u-u_{xx})_x^2(q(x_0,s),s)]ds-(u_x-u_{xxx})(x_0,0).\nonumber\\
\end{eqnarray}
Substituting \eqref{d2} into \eqref{d1}, it yields
\begin{eqnarray}\label{d3}
&&\lefteqn{\frac{d}{dt}[(u-u_{xx})_x^2-(u-u_{xx})^2](q(x_0,t),t)\geq\frac{1}{2}[(u-u_{xx})_x^2-(u-u_{xx})^2]
(q(x_0,t),t)}\nonumber\\
&&\times\left\{\int_0^t[(u-u_{xx})_x^2-(u-u_{xx})^2](q(x_0,s),s)ds+2(u_x-u_{xxx})(x_0,0)\right\}.\nonumber\\
\end{eqnarray}
Before completing the proof, we need the following technical lemma.

\smallskip

\begin{lemma}\emph{\cite{Zhou}}\label{TCD}
Suppose that $\Psi(t)$ is twice continuously differentiable satisfying
\begin{eqnarray}\label{d4}
\left\{
         \begin{array}{ll}
         \Psi''(t)\geq C_0 \Psi'(t)\Psi(t),\quad t>0,\quad C_0>0,\\
         \Psi(t)>0,\quad\Psi'(t)>0.
         \end{array}
       \right.
\end{eqnarray}
Then $\Psi(t)$ blows up in finite time. Moreover, the blow up time can be estimated in terms of the initial datum
as
\begin{eqnarray*}
T\leq \max\left\{\frac{2}{C_0\Phi(0),\frac{\Phi(0)}{\Phi'(0)}}\right\}.
\end{eqnarray*}
\end{lemma}

\smallskip

Let $\Psi(t)=\int_0^t[(u-u_{xx})_x^2-(u-u_{xx})^2](q(x_0,s),s)ds+2(u_x-u_{xxx})(x_0,0)$, then \eqref{d3} is an
equation of type \eqref{d4} with $C_0=\frac{1}{2}$. The proof is completed by applying Lemma \ref{TCD} 
and the boundedness of $u_x$.
\end{proof}

\section{Final Remarks}
{\color{black}
In this final section we collect three different remarks: first, we introduce some $(2n+1)th$ order CH-type equations, $n \geq 1$; second, we discuss the relation of these equations with the geometry of the diffeomorphism
group $Diff(S^1)$; third, we reconnect our class of equations with the geometry of pseudo-spherical surfaces.
}
\subsection{Higher order Camassa-Holm type equations}
In this subsection we proceed as in Section 2. We consider differential operators $A_{2n}$ of order $2n$ and 
define $m = A_{2n}(u)$. Specifically, we choose the operators
\begin{eqnarray}
A_{2n} & = & (-1)^{n+1} \partial_x^{2n} + 2 \sum_{k=1}^{n-1}
(-1)^{n+1-k}\, \partial_x^{2(n-k)} - 1 \; , \label{ion}\\
B_{2n} & = & \sum_{k=0}^{n-1} (-1)^{n-k}\, \partial_x^{2(n-k)-1}\; , \nonumber \\
C_{2n} & = & \sum_{k=0}^{n-1} (-1)^{n-k}\, \partial_x^{2(n-k -1)} \; . \nonumber
\end{eqnarray}
We consider the matrices
\begin{equation} \label{xchn}
X_{2n} = \left[ \begin {array}{cc} 0& \frac{1}{2}\,\lambda +
A_{2n}(u) \\ \noalign{\medskip}\frac{1}{2}\,{\lambda}^{-1}&0\end
{array} \right]
\end{equation}
and
\begin{equation} \label{tchn}
T_{2n} = \left[ \begin {array}{cc} \frac{1}{2}\, B_{2n}(u) &
C_{2n}(u) \, A_{2n}(u) - \frac{1}{2}\, \lambda\, C_{2n}(u) -
\frac{1}{2}\, \lambda^2
\\\noalign{\medskip} \displaystyle -\frac{1}{2} + \frac{1}{2 \lambda}\, C_{2n}
(u) & -\frac{1}{2}\, B_{2n}(u) \end {array} \right] \; .
\end{equation}

A straightforward computation allows us to check that the equation
$$X_{2n,t} - T_{2n,x} + [X_{2n} , T_{2n}] = 0$$ is equivalent to the
{\em $(2n+1)$-order equation of Camassa-Holm type}
\begin{gather}
A_{2n,t}(u)  - 2\, C_{2n}(u)_x\,A_{2n}(u) - C_{2n}(u)\,A_{2n}(u)_x = 0 \; . \label{ch2n}
\end{gather}
It follows from our discussions in Section 1 and \cite{reyes:ChT} that these equations all describe 
pseudo-spherical surfaces. 
We compute quadratic pseudo-potentials and conservation laws as in Section 2: we obtain the Riccati equation
\begin{equation} \label{gcl}
{\frac {\partial \Gamma}{\partial x}}
 = \frac{1}{2}\,\lambda + A_{2n}(u) - {\frac {\Gamma^2}{2\,\lambda}} \; ,
\end{equation}
a corresponding equation for $\Gamma_t$, and the conserved density
$- \Gamma/\lambda$. It follows by expanding $\Gamma$ as in
(\ref{reyes:cd-1})--(\ref{reyes:cd2}) that Equation (\ref{ch2n})
is integrable. We will consider its pseudo-peakon solutions and Cauchy problem elsewhere. 

\subsection{Camassa-Holm type equations and $\mathbf{Diff(S^1)}$}

{\color{black}
In this subsection we connect the (periodic case of the) Camassa-Holm type equations with the geometry of
$Diff(S^1)$, the Fr\'echet Lie group of diffeomorphisms of the circle. We recall some basic facts on the geometry
of this group (some of them already mentioned in Section 1) following the exposition appearing in \cite{GPoR}:

We set $G = Diff(S^1)$ and we write its Lie algebra as $\mathfrak{g}= Vect(S^1)$, see \cite{KW}. We also denote by
$\mathfrak{g}'$ the dual of $\mathfrak{g}$, and
by $\langle\  \ ,\ \rangle : \mathfrak{g} \times \mathfrak{g}' \to \mathbb{R}$
the pairing that induces such a duality.
Given a linear map $A : \mathfrak{g} \to \mathfrak{g}'$ we define the $\mathbb{R}$-bilinear
mapping $(\ \cdot\ \vert\ \cdot\ )_{A} : \mathfrak{g} \times \mathfrak{g} \to \mathbb{R}$ as
$ (\ X\ \vert\ Y\ )_{A} = \langle\ X, A Y\ \rangle $
whenever $X$ and $Y$ are in $\mathfrak{g}$. If such a bilinear
map is symmetric and non-degenerate, we say that $A$ is an {\em inertia operator}. In this case,
we define an adjoint representation with respect to $A$ by
\begin{equation} \label{acc}
(\ \text{ad}(X) Y\ \vert\ Z\ )_{A} = - (\ Y\ \vert\ \text{ad}_{A}(X) Z\ )_{A}
\end{equation}
for all $X, Y, Z$ in $\mathfrak{g}$.

Let us fix an inertia operator $A : \mathfrak{g} \to \mathfrak{g}'$. This operator induces a (pseudo-)Riemannian
metric on $G$: we let $r_{\gamma}$ be the right translation by
$\gamma \in G$, and we denote by $r_{\gamma \ast} : T_{\sigma} G \to T_{\sigma (\gamma)} G$ the induced map
on the tangent bundle. We define the (pseudo-)Riemannian metric induced by $A$ as
\begin{equation} \label{alfa}
(\ \xi\ \vert\ \tau\ )_A(\gamma) =
(\ r_{\gamma^{-1} \ast} \xi\ \vert\ r_{\gamma^{-1} \ast} \tau\ )_{A}
\end{equation}
for all $\tau, \xi \in T_\gamma G$. Now let us consider a smooth curve $\{\ \gamma(t)\ \vert\ t \in T\ \}$ in $G$,
where $T$ is an open interval in $\mathbb{R}$, and let $\dot{\gamma}(t) \in T_{\gamma(t)} G$, $t \in T$, be its
velocity vector. Then, $r_{\gamma(t)^{-1} \ast} \dot{\gamma}(t) = X(t)$ is in $\mathfrak{g}$, and we get a curve
$\{\ X(t)\ \vert\ t \in T\ \}$ in $\mathfrak{g}$. The {\em Euler equation} for $X(t)$ is
\begin{equation}   \label{euler}
  \frac{d}{dt} X(t)  = - \text{ad}_{A}(X(t)) X(t) \; .
\end{equation}
Euler's equation determines geodesics on $G$ with respect to the (pseudo-)Riemannian metric (\ref{alfa}), see
\cite{A} and also \cite{C_fourier,M}: if $X(t)$ solves (\ref{euler}), then the curve $\gamma(t)$ determined by
$r_{\gamma(t)^{-1} \ast} \dot{\gamma}(t) = X(t)$ is a geodesic on $G$.

{\color{black} Let $x$ in} $[0, 2 \pi [$ be the standard coordinate in $S^{1}$.
Every smooth vector field on $S^{1}$ can be written as $ X(x) \partial_{x}$, where
$X : S^{1} \to \mathbb{R}$ is a smooth function. The Lie bracket
between $X = X(x) \partial_{x}$ and $Y = Y(x) \partial_{x}$
is given by $[ X , Y ] = ( X\, Y_x - X_x\, Y)(x)\, \partial_{x}$.
We complexify $\mathfrak{g} = \text{Vect}(S^{1})$, that is we set
$$ \text{Vect}(S^{1})_{\mathbb{C}} = \mathfrak{g} \oplus i\,\mathfrak{g} =
\text{Vect}(S^{1}) \oplus i\, \text{Vect}(S^{1}) \ ,  $$
where $i = \sqrt{-1}$. Thus if $z = e^{i\,x}$, then
$ \{\ l_{n} = z^{n} \partial_{x} \ \vert\ n \in \mathbb{Z} \ \}$ is a basis for
$\text{Vect}(S^{1})_{\mathbb{C}}$,
i.e. for every {\it complex} vector field of the form $ X(x) \partial_{x}$  we have a Fourier
decomposition
$$  X(x) = \sum_{n \in \mathbb{Z}} X_{n} z^{n}\partial_x\; , X_n \in \mathbb{C} \; .  $$
We note that if we set $L_{n} = i l_{n}$, the collection $\{\ L_{n}\ \vert\ n \in \mathbb{Z} \ \}$
is also a basis for $\text{Vect}(S^{1})_{\mathbb{C}}$ and that if we extend the Lie bracket
linearly to $\text{Vect}(S^{1})_{\mathbb{C}}$  we have the
relations $[L_{m} , L_{n}] = (m-n) L_{m+n}$, $m,n \in \mathbb{Z}$.

{\color{black} There is} a non-degenerate, positive-definite, $L_{2}-$ inner product on $\mathfrak{g}$:
if $X = X(x) \partial_{x}$ and $Y = Y(x) \partial_{x}$, then
$$ \langle X\, ,\, Y\rangle = \int_{S^{1}} X(x) Y(x) d\,x  \ .   $$
{\color{black} We use this product to single-out a convenient dual space $\mathfrak{g}^{\prime}$ as in the
beginning of this subsection.}
Extending such a product complex-linearly to $\text{Vect}(S^{1})_{\mathbb{C}}$ we have
$$
\langle l_{m}\, , \, l_{n}\rangle  = 2 \pi \delta_{m,-n} = - \langle L_{m}\, ,\, L_{n}\rangle\; ,
\quad \quad (m,n) \in  \mathbb{Z}^{2}\; .
$$
We fix a finite sequence of real numbers $\mathfrak{c} = \{c_k\}_{k=0}^N$ for some $N \in \mathbb{N}$ and we set
\begin{equation} \label{ac}
 A_{\mathfrak{c}} = \sum_{k=0}^N (-1)^{k} c_{k}\, \partial^{2k}_{x} \; .
\end{equation}
We observe that
$$ (\ X\ \vert\ Y\ )_{A_{\mathfrak{c}}} = \langle X\, ,\, A_{\mathfrak{c}} Y\rangle =
\sum_{k =0}^N c_{k}\, \langle \partial^{k}_{x} X\, ,\, \partial^{k}_{x} Y\rangle =
 \langle A_{\mathfrak{c}} X\, ,\,  Y\rangle $$
for every $X$ and $Y$, and therefore in terms of the basis for $\text{Vect}(S^{1})_{\mathbb{C}}$ we have
$$
(\ l_{m}\ \vert\ l_{n}\ )_{A_{\mathfrak{c}}} =  \sum_{k =0}^N (-1)^k  c_{k}\ m^{k}n^k \langle l_{m}\, , \, l_{n}
\rangle = 2 \pi \delta_{m,-n} \sum_{k =0}^N  c_{k}\ m^{2k}  \; ,
$$
for $m,n \neq 0$, while $(\ l_{0}\ \vert\ l_{n}\ )_{A_{\mathfrak{c}}} = c_0 \langle \partial_x , z^n
\partial_x \rangle = 2 \pi c_0 \delta_{n,0}$.
It follows easily that {\em a symmetric operator $A_{\mathfrak{c}}$ is an inertia operator if and only
if $(\ l_{m}\ \vert\ l_{-m}\ )_{A_{\mathfrak{c}}}$ is different from zero for
every $m$ in $\mathbb{Z}$}.

\smallskip

We are ready to study the geometric interpretation of our equations (\ref{ch2n})\,. First of all, we note that
our operators $A_{2n}$ introduced in (\ref{ion}), are indeed inertia operators. We write $A_{2n}$ as
$$
A_{2n} = (-1)^n (-1) \partial_x^{2n} + \sum_{j=1}^{n-1} (-1)^j (-2) \partial_x^{2j} - 1\; ,
$$
so that, using the notation introduced in (\ref{ac}) we have $c_0=-1$, $c_j =-2$ for $j=1, \cdots,n-1$,
$c_n =-1$. Then, $(\ l_{0}\ \vert\ l_{0}\ )_{A_{\mathfrak{c}}} \neq 0$ ,
$(\ l_{\pm 1}\ \vert\ l_{\pm 1}\ )_{A_{\mathfrak{c}}} \neq 0$, and for $m \neq 0, \pm 1$,
$$
\sum_{k =0}^n  c_{k}\ m^{2k} = -1 - m^{2n} -2 \left( \frac{m^{2n}-1}{m^2 -1} - 1 \right) =
1 - m^{2n} - 2 \, \frac{m^{2n}-1}{m^2 -1}
$$
which is different from zero as well. Now we compute $ad_{A_{2n}}$ using (\ref{acc}); we use `` $'$ " to indicate
derivative with respect to $x$:
\begin{eqnarray*}
\langle [X,Y],A_{2n}(Z) \rangle & = & \int_{S^1} (XY' - X'Y)A_{2n}(Z) dx
\; = \; - \int_{S^1} Y \left[ (X A_{2n}(Z))' + X' A_{2n}(Z) \right] dx \\
& = & - \langle Y , (X A_{2n}(Z))' + X' A_{2n}(Z) \rangle \; = \;
- \langle Y , A_{2n} A_{2n}^{-1}\{(X A_{2n}(Z))' + X' A_{2n}(Z) \} \rangle\; ,
\end{eqnarray*}
so that (\ref{acc}) yields
\begin{equation} \label{acc1}
ad_{A_{2n}}(X)Z = A_{2n}^{-1} \left\{ (X A_{2n}(Z))' + X' A_{2n}(Z) \right\}
= A_{2n}^{-1} \left\{ 2 X' A_{2n}(Z) + X A_{2n}(Z)' \right\} \; .
\end{equation}
This formula implies that Equation (\ref{euler}) in the case $n=1$ is precisely the Camassa-Holm equation,
while the case $n=2$ gives the equation
$$
(-\partial_x^4 + 2 \partial_x^2 - 1)X_t - 2 X_x X_{xxxx} +4 X_x X_{xx} - 3X X_x - X X_{xxxxx} + 2 X X_{xxx} =0\; ,
$$
which {\em is not} our fifth order CH-type equation (\ref{ch410}). In order to find (\ref{ch410}) in the
present framework we write it as in (\ref{CHtype}), this is,
\begin{equation} \label{ch4'}
2 \,(-X_{xx} + X)_x\, m + m_t + (-X_{xx} + X)\, m_x = 0 \; , \quad \quad  m = A_4(X) \; .
\end{equation}
Comparing (\ref{ch4'}) and (\ref{acc1}), we see that our fifth order CH-type equation can be written, in
geometrical terms, as
\begin{equation*} \label{fcht}
{\color{black} \frac{d}{dt}\,X = - ad_{A_4}(A_2(X))\cdot X \; .  }
\end{equation*}
{\color{black}
%This equation is a Hamiltonian equation with respect to the standard Poisson bracket on
%$\mathfrak{g}^{\prime}$ with Hamiltonian function
%$$
%H(X) = - \frac{1}{2}\int_{S^1} (X^2 + X_x^2) dx \; .
%$$
More generally, Equation (\ref{acc1}) implies that our $(2n+1)$th order CH-type equation (\ref{ch2n}) can be
written as
\begin{equation*}
{\color{black} \frac{d}{dt}\,X = ad_{A_{2n}}(C_{2n}(X))\cdot X \; .   }
\end{equation*}
%This equation is also Hamiltonian; the corresponding Hamiltonian function is
%$$
%H(X) =  \frac{1}{2}\int_{S^1} (X^2 + (\partial_x X)^2 + \cdots + (\partial_x^{n-1} X)^2\, ) dx \; .
%$$
}
}

\subsection{Camassa-Holm type equations and classical theory of surfaces}

In this final subsection we go back to our geometric discussion of Section 1. Our equations (\ref{ch4'}) and  
(\ref{ch2n}) describe surfaces of constant Gaussian curvature $K=-1$. As we have already pointed out, 
{\color{black}
%in the following sense, see Chern and Tenenblat's \cite{reyes:ChT} or the later review \cite{reyes:keti}:
%
%\noindent If $\omega^{1} = (X_{2n\;21} + X_{2n\;12})dx +
%(T_{2n\;21} + T_{2n\;12})dt$, $\omega^{2} = 2 X_{2n\;11}dx + 2 T_{2n\;11}dt$, and
%$\omega^{3} = (X_{2n\;21} - X_{2n\;12})dx + (T_{2n\;21} - T_{2n\;12})dt$, in which $X_{2n}$ and $T_{2n}$ are given
%by (\ref{xchn}) and (\ref{tchn}) and $X_{2n\;ij}$, $T_{2n\;ij}$ denote the $(i,j)$ entry of
%$X_{2n}$ and $T_{2n}$ respectively, then
%\begin{equation}
%d {\omega}^{1} = {\omega}^{3} \wedge {\omega}^{2}, ~ ~ ~ ~ d {\omega}^{2} = {\omega}^{1}
%\wedge {\omega}^{3}, ~ ~ ~ ~ \mbox{ and } ~ ~ ~ ~
%d {\omega}^{3} = {\omega}^{1} \wedge {\omega}^{2}  \label{structure0}
%\end{equation}
%whenever $u(x,t)$ solves (\ref{ch2n}). If $\omega^1(u(x,t)) \wedge \omega^2(u(x,t)) \neq 0$, these structure
%equations say that the domain of $u(x,t)$ has the structure of a surface of constant Gaussian curvature equal to
%$-1$, with metric $({\omega}^{1})^{2} + ({\omega}^{2})^{2}$ and
%connection one--form ${\omega}_{12} = {\omega}^{3}$.
%
the importance of this observation is that the methods used in Section 2 for obtaining conservation laws and
pseudo-potentials originate within the geometry of pseudo-spherical surfaces, as noted by Chern and Tenenblat in
\cite{reyes:ChT}. Later papers on these topics are \cite{er_cl,reyes:keti}, and \cite{reyes:R4} for the particular
case of the Camassa-Holm equation.

A recent endeavour within the theory of equations describing pseudo-spherical surfaces is to
investigate local isometric immersions into $E^3$ of the pseudo-spherical surfaces described intrinsically by
solutions of equations such as (\ref{ch2n}). It is a classical result that every pseudo-spherical surface can
be locally immersed in $E^3$, and that the existence of such immersion is due to the fact that one can find
further one-forms
$$
\omega_{13} = a\, \omega_1 + b\, \omega_2\; , \quad \quad \omega_{23} = b\, \omega_1 + c\, \omega_2\; ,
$$
satisfying the equations
\begin{equation} \label{st20}
d \omega_{13} = \omega_{12} \wedge \omega_{23}\; , \quad \quad d \omega_{23} = -\omega_{12} \wedge \omega_{13}
\; , \quad \quad a\,c - b^2 = -1 \; .
\end{equation}
Generally speaking, the functions $a,b,c$ are determined by solving differential equations, and therefore
it is quite surprising that (for one-forms associated to differential equations describing pseudo-spherical
surfaces) in some cases they can be expressed in closed form as functions of the independent
variables and of a {\em finite number} of derivatives of the dependent variable $u$, see for instance
\cite[p. 1650021-4]{CSK}. Even more, it has been noticed that in some important instances ({\em e.g.} the
Degasperis-Procesi equation), the functions $a,b,c$
{\em depend  only on the independent variables $x,t$}, see \cite[Theorem 1.1]{CSK}, while (up to a technical
condition) it is not possible to find functions $a,b,c$  depending on $x,t$ and at most a finite number of
derivatives of $u$ in the case of the Camassa-Holm equation. Thus, it is very natural to ask whether or not we can
find functions $a,b,c$ depending on finite order jets for our equations (\ref{ch2n}). Our first result
concerns the fifth order CH-type equation (\ref{ch410}).

\begin{proposition}
\begin{enumerate}
\item
The fifth order CH-type equation $(\ref{ch410})$ describes pseudo-spherical surfaces with associated one-forms
\begin{eqnarray}
\omega^1 = \left(\frac{\lambda}{2}-m +\frac{1}{2 \lambda}\right)dx +\left(v m +\frac{\lambda  v }{2}-
\frac{\lambda^{2}}{2}-\frac{1}{2}-\frac{v}{2 \lambda}\right)dt \; , \quad \quad \omega^2 = -v_{x} dt\; , \quad
\label{st301} \\
\omega^3 = \left(-\frac{\lambda}{2}+m +\frac{1}{2 \lambda}\right)dx +\left(-v m -\frac{\lambda  v}{2}+
\frac{\lambda^{2}}{2}-\frac{1}{2}-\frac{v}{2 \lambda}\right)dt \;  \quad \label{st302}
\end{eqnarray}
in which
$$
v = u - u_{xx} \; , \quad \quad m = v - v_{xx}\; .
$$
\item There are no functions $a,b,c$ depending only on independent variables $x,t$ such that the one forms
$\omega^1, \omega^2, \omega_{12} = \omega^3$ given by $(\ref{st301})$, $(\ref{st302})$ and
$$
\omega_{13} = a\, \omega_1 + b\, \omega_2\; , \quad \quad \omega_{23} = b\, \omega_1 + c\, \omega_2
$$
satisfy the structure equations $(\ref{o4})$ and $(\ref{st20})$.
\end{enumerate}
\end{proposition}
\begin{proof}
Item 1 is a consequence of Theorem 1 in Section 1. It is also a straightforward direct computation: we simply note 
that Equation (\ref{ch410}) can be written as in (\ref{FOCHInt1}), that is, as the system
\begin{equation}  \label{FOCHInt2}
-2\,m\,v_x - m_x\,v = m_t\; , \quad \quad v = u - u_{xx} \; , \quad \quad m = v - v_{xx}\; ,
\end{equation}
and we check that the one-forms appearing in (\ref{st301}) and (\ref{st302}) satisfy the structure equations
(\ref{o4}). Item 2 is proven by a strategy similar in spirit to the
one used in the proof of Proposition 3: let us assume that there exist functions $a,b,c$ depending only on $x,t$
such that Equations (\ref{st20}) hold. We write down (\ref{st20}) and we obtain two equations that have to be
satisfied identically whenever $u(x,t)$ solves (\ref{FOCHInt2}), the first one corresponding to
$d \omega_{13} = \omega_{12} \wedge \omega_{23}$ and the second one corresponding to
$d \omega_{23} = -\omega_{12} \wedge \omega_{13}$. We will simply call them Equations M and N respectively.
The sketch that follows has been checked using MAPLE 2021:

Taking derivatives of M with respect to $v_x$ and then with respect to $m$ we get $a(x,t)=c(x,t)$. Replacing into
M and taking derivative with respect to $v_x$ we find $b_x =0$, while replacing $a(x,t)=c(x,t)$ into N and then
differentiating  with respect to $v_x$ and $m$ yields $b=0$. Replacing into N once again gives us $c_x=0$, and
then replacing into M we find $c_t=0$. Now we use the Gauss equation $ac-b^2 =-1$ and we conclude that $a,b,c$
cannot exist.
\end{proof}

Now we consider Equations (\ref{ch2n}) in full generality. In order to study the local immersion problem we
proceed differently than in the interesting paper \cite{CSK}. Instead of trying to find one-forms $\omega_{13}$
and $\omega_{23}$ satisfying (\ref{st20}) directly, we simply construct a rather explicit local immersion, taking
advantage of the following observation appearing in \cite{KT}:

\begin{lemma} \label{nk}
There exists a local diffeomorphism $\Psi : V \subseteq \mathbb{R}^2 \rightarrow \overline{W}$, where
$\overline{W}$ is a subset of the Poincar\'e upper half plane, and a smooth function $\mu$ such that
\begin{eqnarray*}
\Psi^\ast \theta^1 & = & \cos\mu\, \omega^1 + \sin \mu\, \omega^2 \\
\Psi^\ast \theta^2 & = & -\sin\mu\, \omega^1 + \cos \mu\, \omega^2 \\
\Psi^\ast \theta^3 & = & \omega^3 + d \mu
\end{eqnarray*}
where
$$
\theta^1 = \frac{d \overline{x}}{\overline{t}}\; , \quad \quad \theta^2 = \frac{d \overline{t}}{\overline{t}}\; ,
\quad \quad \theta^3 = \frac{d \overline{x}}{\overline{t}}\; .
$$
\end{lemma}

The one-forms $\theta^1, \theta^2$ give the standard metric on the Poincar\'e upper half plane (hereafter denoted
by $\mathbb{H}$) $ds^2 = \theta^1 \otimes \theta^1 + \theta^2 \otimes \theta^2$, and $\theta^3$ is the
corresponding connection one-form. The proof of this lemma is in \cite[pp. 90-91]{KT}. Now we note that
 $\mathbb{H}$ can be immersed into $E^3$ explicitly. A well known immersion is given by
the function $F : U \subseteq \mathbb{H} \rightarrow E^3$, $U = \{ (\overline{x}, \overline{t}) \in
\mathbb{H} : \overline{t} > 1 \}$, with
$$
F(\overline{x} , \overline{t} )=( f(\overline{t}) \cos \overline{x}\, ,\, f(\overline{t}) \sin \overline{x}\, ,\,
g(\overline{t}) )\; ,
$$
where
$$
f(\overline{t}) = \frac{1}{\overline{t}} \; \quad \quad g(\overline{t}) =
\ln \left( \sqrt{\overline{t}^2 -1} + \overline{t} \right) - \frac{\sqrt{\overline{t}^2 -1}}{\overline{t}}\; .
$$
Thus, {\em a local isometric immersion from the pseudo-spherical structure on $V$ induced
by our Ch-type equation $(\ref{ch2n})$ into $E^3$ is given by the composition
$\Phi = F \circ \Psi$. } Certainly, this immersion is in principle highly ``nonlocal", since
it depends on the diffeomorphism $\Psi$ that is found by means of the Frobenius theorem, see \cite[p. 91]{KT}.
However, we believe that this nonlocality is interesting in its own right:

The first and third equations appearing in Lemma \ref{nk} imply that we can obtain the function $\mu$ via the
Pfaffian system
$$
\cos\mu\, \omega^1 + \sin\mu\, \omega^2 = \omega^3 + d \mu\; .
$$
The change of variables $\Gamma = \tan(\mu/2)$ transforms this equation into the Riccati system
$$
2 d \Gamma = (\omega^1 - \omega^3) + 2 \Gamma \omega^2 - \Gamma^2 (\omega^1 + \omega^3) \; ,
$$
and using the explicit expressions for the one-forms $\omega^i$, $i=1,2,3$, appearing at the beginning of this
subsection we obtain that this system is equivalent to
$$
2 \Gamma_x = \lambda + 2 A_{2n}(u) - \frac{1}{\lambda}\Gamma^2
$$
and an equation for $\Gamma_t$ which we will not write down. This equation for $\Gamma_x$ is precisely the
quadratic pseudo-potential equation (\ref{gcl}) determining local conservation laws for the CH-type equation
(\ref{ch2n})! Also, we can check that if we write $\Psi(x,t) = (\phi(x,t) , \psi(x,t))$, then
$\ln(\psi)$ is a potential for the local conservation laws of (\ref{ch2n}), while $\phi$ is a further potential
depending on $\psi$.

Thus, local isometric immersions of {\color{black} pseudo-spherical surfaces determined by solutions to
our CH-type equations as explained at the beginning of this subsection} are, essentially, constructed via their local conservation laws.
}

\subsection*{Acknowledgements}
E.G.R's work has been partially supported by the Universidad de
Santiago de Chile through the DICYT grant \# 041533RG and by the FONDECYT grant \#1201894.
The third author  thanks
the UT President's Endowed Professorship (Project \# 450000123).
%The research project is supported by National Natural Science Foundation of China(Grant Nos. 11772007) and also supported by Beijing Natural Science Foundation
%and
%Basic and Interdisciplinary Scientific Research Program (Project \# J202134) of Beijing University of Technology
 %High-level Foreign Experts Funding Program of Beijing (Project \# J202134)  for their partial support.

\end{document}